\documentclass[11pt,a4paper]{amsart}

\newcommand{\C}{\mathbb{C}}
\newcommand{\R}{\mathbb{R}}
\newcommand{\N}{\mathbb{N}}

\usepackage{amssymb,amsmath,amsthm,enumerate}
\usepackage{graphicx}
\usepackage{color}

\usepackage{hyperref}
\usepackage{cite}

\usepackage{etoolbox}
\newtheorem{theorem}{Theorem}[section]
\newtheorem{lemma}[theorem]{Lemma}

\newtheorem{corollary}[theorem]{Corollary}
\newtheorem{definition}[theorem]{Definition}

\preto{\section}{}
\preto{\subsection}{}

\makeatletter
\@addtoreset{theorem}{section}
\@addtoreset{theorem}{subsection}
\makeatother

\usepackage{url}
\begin{document}

\title{Incidence bounds on multijoints and generic joints}
\author{ Marina Iliopoulou}
\address{School of Mathematics, University of Birmingham, Birmingham, Edgbaston, B15 2TT, UK}
\email{\href{mailto:M.Iliopoulou@bham.ac.uk}{M.Iliopoulou@bham.ac.uk}}
\maketitle

\begin{abstract}
A point $x \in \mathbb{F}^n$ is a joint formed by a finite collection $\mathfrak{L}$ of lines in $\mathbb{F}^n$ if there exist at least $n$ lines in $\mathfrak{L}$ through $x$ that span $\mathbb{F}^n$. It is known that there are $\lesssim_n |\mathfrak{L}|^{\frac{n}{n-1}}$ joints formed by $\mathfrak{L}$. 

We say that a point $x \in \mathbb{F}^n$ is a multijoint formed by the finite collections $\mathfrak{L}_1,\ldots,\mathfrak{L}_n$ of lines in $\mathbb{F}^n$ if there exist at least $n$ lines through $x$, one from each collection, spanning $\mathbb{F}^n$. We show that there are $\lesssim_n (|\mathfrak{L}_1|\cdots |\mathfrak{L}_n|)^{\frac{1}{n-1}}$ such points for any field $\mathbb{F}$ and $n=3$, as well as for $\mathbb{F}=\R$ and any $n \geq 3$. 

Moreover, we say that a point $x \in \mathbb{F}^n$ is a generic joint formed by a finite collection $\mathfrak{L}$ of lines in $\mathbb{F}^n$ if each $n$ lines of $\mathfrak{L}$ through $x$ form a joint there. We show that, for $\mathbb{F}=\R$ and any $n \geq 3$, there are $\lesssim_n \frac{|\mathfrak{L}|^{\frac{n}{n-1}}}{k^{\frac{n+1}{n-1}}}+\frac{|\mathfrak{L}|}{k}$ generic joints formed by $\mathfrak{L}$, each lying in $\sim k$ lines of $\mathfrak{L}$. This result generalises, to all dimensions, a (very small) part of the main point-line incidence theorem in $\R^3$  in \cite{Guth_Katz_2010} by Guth and Katz.

Finally, we generalise our results in $\R^n$ to the case of multijoints and generic joints formed by real algebraic curves.
\end{abstract}

\section{Introduction}

A point $x \in \mathbb{F}^n$, where $\mathbb{F}$ is a field and $n \geq 2$, is a \textit{joint} for a finite collection $\mathfrak{L}$ of lines in $\mathbb{F}^n$ if there exist at least $n$ lines in $\mathfrak{L}$ passing through $x$, whose directions span $\mathbb{F}^n$. We denote by $J(\mathfrak{L})$ the set of joints formed by $\mathfrak{L}$. The joints problem asks for the optimal upper bound on $|J(\mathfrak{L})|$, depending only on $|\mathfrak{L}|$, and first appeared in \cite{Chazelle_Edelsbrunner_Guibas_Pollack_Seidel_Sharir_Snoeyink_1992}. After partial progress (see \cite{MR1280600}, \cite{MR2047237}, \cite{MR2121298}, \cite{MR2763049}), it was fully solved by Guth and Katz in $\R^3$ (in \cite{Guth_Katz_2008}), and then in $\R^n$ by Quilodr\'an (in \cite{MR2594983}) and independently by Kaplan, Sharir and Shustin (in \cite{MR2728035}), who showed that\footnote{In whatever precedes and follows, any expression of the form $A \lesssim B$ means that there exists an explicit non-negative constant $M$, such that $A \leq M \cdot B$, while any expression of the form $A\lesssim_{b_1,...,b_m} B$ means that there exists a non-negative constant $M_{b_1,...,b_m}$, depending only on $b_1$, ..., $b_m$, such that $A \lesssim M_{b_1,...,b_m}\cdot B$. In addition, any expression of the form $A \gtrsim B$ or $A \gtrsim_{b_1,...,b_m} B$ means that $B \lesssim A$ or $B \lesssim_{b_1,...,b_m} A$, respectively. Finally, any expression of the form $A \sim B$ means that $A \lesssim B$ and $A \gtrsim B$, while any expression of the form $A \sim_{b_1,...,b_m} B$ means that $A \lesssim_{b_1,...,b_m} B$ and $A \gtrsim_{b_1,...,b_m} B$.}
\begin{equation}|J(\mathfrak{L})| \lesssim_n |\mathfrak{L}|^{\frac{n}{n-1}}. \label{eq:basic}
\end{equation}
All solutions to the joints problem are based on the polynomial method, which was introduced in the area by Dvir for the solution of the Kakeya problem in finite fields (see \cite{MR2525780}). Note that \eqref{eq:basic} is known in fact in any field setting (for example, see \cite{Dvir_12}, \cite{Tao_14}); the proof is similar to the one in Euclidean space. Joints are interesting in their own right, playing an important role in incidence geometry, but also have connections with harmonic analytic problems (see the last subsection of the Introduction).

We now vary the notion of a joint, to form the notion of a multijoint: a joint formed by $n$ collections of lines in $\mathbb{F}^n$. 

\begin{definition} Let $\mathfrak{L}_1, \ldots , \mathfrak{L}_n$ be finite collections of lines in $\mathbb{F}^n$. We say that a point $x \in \mathbb{F}^n$ is a \textit{multijoint} formed by the $n$ collections of lines if, for each $i=1,\ldots, n$, there exists a line $l_i \in \mathfrak{L}_i$ passing though $x$, so that the directions of $l_1, \ldots, l_n$ span $\mathbb{F}^n$. We denote by $J(\mathfrak{L}_1,\ldots,\mathfrak{L}_n)$ the set of multijoints formed by $\mathfrak{L}_1,\ldots,\mathfrak{L}_n$.
\end{definition}

Naturally, the question arises of whether there exists a multilinear analogue of \eqref{eq:basic} in the case of multijoints. \\ \\
\textbf{Question 1. (Carbery)} \textit{Is it true that, for any field $\mathbb{F}$ and any $n \geq 3$, $|J(\mathfrak{L}_1,\ldots,\mathfrak{L}_n)| \lesssim_n (L_1\cdots L_n)^{\frac{1}{n-1}}$, for all finite collections $\mathfrak{L}_1, \ldots, \mathfrak{L}_n$ of $L_1,$ $\ldots, L_n$, respectively, lines in $\mathbb{F}^n$?} \\

It is conjectured (Carbery) that the answer to this question is positive. In fact, estimates suggestive of this have been shown by Carbery and Valdimarsson in \cite{CarberyValdimarsson}, in any field setting and dimension. Moreover, we have answered Question 1 in the affirmative in the case of $\R^3$ in \cite{Iliopoulou_14}. Our proof in \cite{Iliopoulou_14}, however, makes use of Guth--Katz polynomial partitioning (which relies on properties of Euclidean space), and the fact that the number of critical lines contained in an algebraic hypersurface in $\R^3$ is bounded from above, which does not necessarily hold in higher dimensions. Here, we answer Question 1 in the affirmative in the case of $\R^n$, for all $n \geq 3$ (for $n=2$ it is obvious in any field setting), as well as in the case of $\mathbb{F}^3$, where $\mathbb{F}$ is any field. In particular, we show the following.

\begin{theorem}\label{multijointsarbitrary} Let $\mathfrak{L}_1$, $\mathfrak{L}_2$, $\mathfrak{L}_3$ be finite collections of $L_1$, $L_2$ and $L_3$, respectively, lines in $\mathbb{F}^3$, where $\mathbb{F}$ is an arbitrary field. Then,
$$|J(\mathfrak{L}_1,\mathfrak{L}_2,\mathfrak{L}_3)| \leq c \;(L_1L_2L_3)^{1/2},$$
where $c$ is an absolute constant.

\end{theorem}

\begin{theorem}\label{theoremmult2} Let $n\geq 2$. Let $\mathfrak{L}_1,\dots, \mathfrak{L}_n$ be finite collections of $L_1, \dots, L_n$, respectively, lines in $\R^n$. Then,
\begin{equation} \label{eq:proved}|J(\mathfrak{L}_1,\dots, \mathfrak{L}_n)|\leq c_n (L_1\cdots L_n)^{1/(n-1)}, \end{equation}
where $c_n$ is a constant depending only on $n$.
\end{theorem}

We also generalise Theorem \ref{theoremmult2} to the case of multijoints formed by real algebraic curves. 

Our techniques are based on the polynomial method. More particularly, we use a probabilistic polynomial degree reduction argument in the case of $\mathbb{F}^3$, when $\mathbb{F}$ is an arbitrary field, and Guth--Katz partitioning in the case of $\R^n$ (see an informal discussion in the subsection that follows).\\

Now, let us go back to joints. 

\begin{definition} We say that a point $x \in \mathbb{F}^n$ is a \emph{generic joint} formed by a finite collection $\mathfrak{L}$ of lines in $\mathbb{F}^n$ if, whenever $n$ lines in $\mathfrak{L}$ pass through $x$, they form a joint there. For all $k \geq n$, we denote by $J^k(\mathfrak{L})$ the set of generic joints formed by a finite collection $\mathfrak{L}$ of lines in $\mathbb{F}^n$, each of which lies in at least $k$ and fewer than $2k$ lines of $\mathfrak{L}$.
\end{definition}

We want to know whether one can exploit the property of genericity, to find better estimates on the size of sets of generic joints than the ones that hold in the non-generic situation. 

Let us make a first attempt to answer this question. A standard probabilistic argument (see \cite{Carbery_Iliopoulou_14}) ensures that, thanks to the genericity hypothesis, \eqref{eq:basic} implies that \begin{equation} \label{eq:genericprobabilistic}|J^k(\mathfrak{L})| \lesssim_n \frac{L^{\frac{n}{n-1}}}{k^{\frac{n}{n-1}}},\end{equation}
for any finite collection $\mathfrak{L}$ of $L$ lines in $\R^n$, and any $k\geq n$ (indeed, by randomly choosing each line through each point of $J^k(\mathfrak{L})$ with probability $1/k$, we see that there exists a subcollection $\widetilde{\mathfrak{L}}$ of $\mathfrak{L}$, consisting of $\sim_n L/k$ lines, such that each point in a large proportion of $J^k(\mathfrak{L})$ lies in at least $n$ of these lines; now, due to the genericity hypothesis, the points in this large proportion of $J^k(\mathfrak{L})$ are joints formed by $\widetilde{\mathfrak{L}}$, and thus, by \eqref{eq:basic}, they number $\lesssim_n |\widetilde{\mathfrak{L}}|^{\frac{n}{n-1}}\sim_n \frac{{L}^{\frac{n}{n-1}}}{k^{\frac{n}{n-1}}}$.)

However, in $\R^2$, the Szemer\'edi--Trotter theorem (see next section) implies that
$$|J^k(\mathfrak{L})|\lesssim \frac{L^2}{k^3}+\frac{L}{k},
$$
which is a better estimate than \eqref{eq:genericprobabilistic} (note that the Szemer\'edi--Trotter theorem fails in finite-field settings). Moreover, an immediate consequence of the proof of the main point-line incidence theorem in \cite{Guth_Katz_2010} by Guth and Katz is that, in $\R^3$,
$$|J^k(\mathfrak{L})|\lesssim \frac{L^{3/2}}{k^2}+\frac{L}{k},
$$
which, again, is a better estimate than \eqref{eq:genericprobabilistic}. Therefore, the following question arises. \\ \\
\textbf{Question 2.} \textit{Is it true that $|J^{k}(\mathfrak{L})| \lesssim_n \frac{L^{\frac{n}{n-1}}}{k^{\frac{n+1}{n-1}}}+\frac{L}{k}$, for all $k \geq n$, for any collection $\mathfrak{L}$ of $L$ lines in $\R^n$?}\\

(Note that the above-mentioned proof of Guth and Katz makes use of the fact that an algebraic hypersurface in $\R^3$ contains a bounded number of critical lines, a fact which is not always true in higher dimensions.)

Here, we answer Question 2 in the affirmative:

\begin{theorem} \label{genericjoints} Let $n \geq 2$. Let $\mathfrak{L}$ be a finite collection of L lines in $\R^n$. Then,
$$|J^{k}(\mathfrak{L})| \leq c_n \left(\frac{L^{\frac{n}{n-1}}}{k^{\frac{n+1}{n-1}}}+\frac{L}{k}\right)$$ for all $k \geq n$, where $c_n$ is a constant depending only on $n$.
\end{theorem}

We also generalise Theorem \ref{genericjoints} to the case of generic joints formed by real algebraic curves. Our techniques are based on Guth--Katz partitioning.\\

Note that the statement of Theorem \ref{genericjoints} does not always hold in the non-generic case in $\R^n$. Indeed, we now construct a counterexample in $\R^n$, where the joints are points in the lattice square $\{0,1,\ldots, N\}^{n-1}$ in the hyperplane $\{x_n=0\}$, each with a bush of rich (i.e. with a lot of these points) lines in $\{x_n=0\}$ through it, and a line in $\R^n$ through it orthogonal to the hyperplane: 

We say that a point $y \in \mathbb{Z}^d$ is \textit{visible} from another point $x \in \mathbb{Z}^d$ if the line segment in $\R^d$ connecting $x$ with $y$ does not contain any other point of $\mathbb{Z}^d$. For any $N \in \mathbb{Z}$, let $A_d(N)$ be the number of points in the lattice $\{0,1,\ldots,N\}^d$ that are visible from $0 \in \mathbb{Z}^d$. It is known (see \cite{Hardy_Wright}) that $\lim_{N \rightarrow \infty}\frac{A_d(N)}{N^d}=\frac{1}{\zeta(d)}$, where $\zeta$ is the Riemann zeta function. This means that, for $N$ sufficiently large, $\sim_d N^d$ points in $\{0,1,\ldots,N\}^d$ (i.e. pretty much all the points of that square lattice, up to multiplication with constants) are visible from the origin. In particular, for any fixed $\alpha \in (0,1)$ and $N$ large, there are $\sim_n N^{\alpha (n-1)}$ points in the square lattice $\{0,1,\ldots, N^{\alpha}\}^{n-1}$ that are visible from the origin. This means that, on the hyperplane $\{x_n=0\}$ in $\R^n$, there exist $\sim N^{\alpha (n-1)}$ lines through $0 \in \R^n$, each containing $\sim_n 1$ points of $(\{0,1,\ldots, N^{\alpha}\}^{n-1} \setminus \{0\})\times \{0\}$, and thus $\sim_n N^{1-\alpha}$ points of $\{0,1,\ldots, N\}^{n-1}\times\{0\}:=\Lambda$. Let $\mathfrak{L}_0$ be the set of these lines. Now, for each $y \in \Lambda$, we translate each line of $\mathfrak{L}_0$ so that it passes through $y$; let $\mathfrak{L}_y$ be the set of resulting lines. Now, if each line in $\mathfrak{L}_y$ intersects $[0,N]^{n-1}\times \{0\}$ at a line segment of length $\sim_n N$, we set $\mathfrak{L}'_y:=\mathfrak{L}_y$. Otherwise, we set $\mathfrak{L}'_y$ to be the set of lines on $\{x_n=0\}$ that are the reflections of the lines in $\mathfrak{L}_y$ with respect to the hyperplane $\{x_1=y_1\}$ in $\R^n$, where $y_1$ is the first coordinate of $y$; in this case, each of these reflected lines intersects $[0,N]^{n-1}\times \{0\}$ at a line segment of length $\sim_n N$. Therefore, each line in $\mathfrak{L}'_y$ contains $\sim_n N^{1-\alpha}$ points of $\Lambda$. We have therefore constructed a set $\mathfrak{L}'$ of $\sim_n \frac{N^{n-1}N^{\alpha(n-1)}}{N^{1-\alpha}}$ lines in the hyperplane $\{x_n=0\}$ in $\R^n$, such that each point of $\Lambda$ lies in $\sim_n N^{\alpha(n-1)}$ of the lines. Now, for each point of $\Lambda$, we consider a line in $\R^n$ orthogonal to the hyperplane, passing through the point. Let $\mathfrak{L}$ be the union of the set of these lines with $\mathfrak{L}'$. It is clear that $|\mathfrak{L}|\sim_n \frac{N^{n-1}N^{\alpha(n-1)}}{N^{1-\alpha}}+N^{n-1}$, $J(\mathfrak{L})=\Lambda$, and each point of $J(\mathfrak{L})$ lies in $\sim_n k :=N^{\alpha(n-1)}$ lines of $\mathfrak{L}$. And it is easy to see that, for $\alpha \in (\frac{1}{n+1}, 1)$, it does not hold that $|J(\mathfrak{L})| \lesssim_n \frac{L^{\frac{n}{n-1}}}{k^{\frac{n+1}{n-1}}}+\frac{L}{k}$. \\ \\
\textit{\textbf{Strategy.}} Our proofs will be carried out with the use of the polynomial method. Here we give a very rough description of our main idea. It is based on the following standard lemma, whose version in $\R^n$ was Quilodr\'an's basic argument for the solution of the joints problem in \cite{MR2594983}:

\begin{lemma}\label{firstpoly} Let $\mathbb{F}$ be a field, $n \geq 2$. Let $J$ be a set of joints formed by a collection $\mathfrak{L}$ of lines in $\mathbb{F}^n$. If $J$ lies in the zero set of some non-zero polynomial $p \in \mathbb{F}[x_1,\ldots,x_n]$, then there exists a line in $\mathfrak{L}$ containing at most $\deg p$ points of $J$.

\end{lemma}

In particular, Quilodr\'an deduced \eqref{eq:basic} in \cite{MR2594983}, combining Lemma \ref{firstpoly} for $\mathbb{F}=\R$ with the following Lemma, which was used by Dvir in \cite{MR2525780} for the solution of the Kakeya problem in finite fields. 

\begin{lemma} \label{dvir} Let $\mathbb{F}$ be a field, $n \geq 2$. For any finite set $\mathcal{P}$ of points in $\mathbb{F}^n$, there exists a non-zero polynomial $p \in \mathbb{F}[x_1,\ldots,x_n]$, of degree $\lesssim_n |\mathcal{P}|^{1/n}$, vanishing on $\mathcal{P}$.
\end{lemma}

Quilodr\'an's proof of \eqref{eq:basic} consists in applying Lemma \ref{dvir} for a set of joints in $\R^n$, taking out of the set of lines the one that contains few joints, together with the joints it contains, and continuing iteratively until all the joints are removed; eventually we see that in total the joints cannot be too many, since in each step only few joints were removed. More precisely, this iterative process, which has been generalised in any field setting (see \cite{Dvir_12}, \cite{Tao_14}), shows that, for any finite collection $\mathfrak{L}$ of lines in $\mathbb{F}^n$, $|J(\mathfrak{L})|$ is equal to at most the number of steps in the iterative process times the maximal number of joints that can be removed in each step, i.e. $|J(\mathfrak{L})|  \lesssim_n |\mathfrak{L}| |J(\mathfrak{L})|^{1/n}$, and thus $|J(\mathfrak{L})| \lesssim_n |\mathfrak{L}|^{\frac{n}{n-1}}$.

When counting multijoints, we would like to again exploit Lemma \ref{firstpoly}, and deduce our desired estimates via a similar iterative process. In particular, if $\mathfrak{L}_1,\ldots, \mathfrak{L}_n$ are finite collections of $L_1,\ldots, L_n$, respectively, lines in $\mathbb{F}^n$, with $L_1\leq\ldots\leq L_n$, then we know by Lemma \ref{dvir} that there exists a non-zero polynomial $p$, of degree $\lesssim_n |J(\mathfrak{L}_1,\ldots,\mathfrak{L}_n)|^{1/n}$, vanishing on $J(\mathfrak{L}_1,\ldots,\mathfrak{L}_n)$. However, using such a polynomial and applying an iterative process as in the case of joints with the help of Lemma \ref{firstpoly}, we get that $|J(\mathfrak{L}_1,\ldots,\mathfrak{L}_n)|\leq \deg p \;(L_1+\ldots+L_n)\lesssim_n |J(\mathfrak{L}_1,\ldots,\mathfrak{L}_n)|^{1/n} (L_1+\ldots +L_n)$, and so $|J(\mathfrak{L}_1,\ldots,\mathfrak{L}_n)| \lesssim_n (L_1+\ldots+L_n)^{\frac{n}{n-1}}$, an upper bound worse than $(L_1\cdots L_n)^{\frac{1}{n-1}}$. The reason is that the degree of our vanishing polynomial is too large for us to get our desired bound in this case. So, we need to find a polynomial of degree considerably lower than $|J(\mathfrak{L}_1,\ldots,\mathfrak{L}_n)|^{1/n}$, whose zero set vanishes on at least a large proportion of our set of multijoints, and then hope that the iterative process will give the desired estimate. \footnote{Note that there exists a polynomial in $\mathbb{F}[x_1,\ldots, x_n]$, of degree $\lesssim L_1^{1/(n-1)}$, vanishing on all the lines of $\mathfrak{L}_1$, and thus on $J(\mathfrak{L}_1,\ldots,\mathfrak{L}_n)$ (see Lemma \ref{guth}). However, this degree is again too large to imply the estimate we want; we would only get $|J(\mathfrak{L}_1,\ldots,\mathfrak{L}_n)|\lesssim_n L_1^{1/(n-1)}L_n$.} In the case of multijoints in $\mathbb{F}^3$, this will be achieved via a probabilistic polynomial degree reduction argument, similar to the ones appearing in \cite{Guth_Katz_2010} and \cite{Guth_14}. In the case of multijoints in $\R^n$, it will be done using Guth--Katz polynomial partitioning. More precisely, we will show that, if the desired multijoints estimate holds in one dimension lower (which implies an estimate on our multijoints involving $n-1$ of our collections of lines), then a large proportion of our set of multijoints will be lying on the zero set of a polynomial of low enough degree; the fact that our desired estimate obviously holds in $\R^2$ allows us to close the induction.

In the case of generic joints in $\mathbb{F}^n$, we have shown in \cite{Carbery_Iliopoulou_14} that, using Lemma \ref{firstpoly} and Quilodr\'an's argument above, but counting each joint as many times as the lines passing through it, we only get \eqref{eq:genericprobabilistic} (note that a generic joint formed by a collection of lines remains a joint after the removal of any line through it, as long as there remain at least $n$ lines passing through it). Therefore, we again wish to find a non-zero polynomial of considerably lower degree, vanishing on at least a large proportion of our set of joints; and then this Quilodr\'an-style argument will imply the estimate we want. We will find such a polynomial in $\R^n$ using Guth--Katz partitioning. Similarly to the multijoints case, we will show that, if our desired estimate holds in one dimension lower, then a large proportion of our set of joints will lie on the zero set of a non-zero polynomial of low enough degree. The fact that our desired estimate holds in $\R^2$ (by the Szemer\'edi--Trotter theorem) closes the induction.\\ \\
\textit{\textbf{Connection with Kakeya conjectures.}} We would like to conclude this section by explaining the connection of joints and multijoints with Kakeya conjectures in harmonic analysis. In particular, the joints and multijoints problems are discrete analogues of, respectively, the maximal Kakeya operator conjecture and the endpoint multilinear Kakeya problem (the latter was solved by Guth in \cite{Guth_10}):\\ \\
\textbf{Maximal Kakeya operator conjecture.} \textit{Let $T_{\omega}$, for $\omega \in \Omega \subset \mathbb{S}^{n-1}$, be a tube in $\R^n$, with direction $\omega$, length $1$ and cross-section an $(n-1)$-dimensional ball of radius $\delta$. If the set $\Omega$ of directions is a $\delta$-separated subset of $\;\mathbb{S}^{n-1}$, then}
\begin{equation}\label{eq:maximalkakeya}\int_{\R^n}\#\{T_{\omega}\text{'s through }x\}^{\frac{n}{n-1}}\;{\rm{d}} x \lesssim_n \log \left(\frac{1}{\delta}\right)\cdot \delta^{n-1}\#\{T_{\omega}\text{'s}\}.
\end{equation}

The endpoint multilinear Kakeya problem is a multilinear version of the maximal Kakeya operator conjecture, and was solved by Guth in \cite{Guth_10}, who improved the already existing result by Bennett, Carbery and Tao (see \cite{MR2275834}).\\ \\
\textbf{Endpoint multilinear Kakeya theorem. (Guth, \cite{Guth_10})} \textit{Let $\mathbb{T}_1$, ..., $\mathbb{T}_n$ be $n$ essentially transversal \footnote{The expression ``essentially transversal" means that, for all $i=1,...,n$, the direction of each tube in the family $\mathbb{T}_i$ lies in a fixed $\frac{c}{n}$-cap around the vector $e_i \in \R^n$, where the vectors $e_1$, ..., $e_n$ are orthonormal.} families of doubly-infinite tubes in $\mathbb{R}^n$, with cross section an $(n-1)$-dimensional unit ball. Then,}
$$\int_{x \in \R^n}\big(\#\{\text{tubes of }\mathbb{T}_1\text{ through }x\}\cdots \#\{\text{tubes of }\mathbb{T}_n\text{ through }x\}\big)^{1/(n-1)}{\rm d}x $$
\begin{equation} 
\label{eq:guthstrong}\lesssim_n (|\mathbb{T}_1| \cdots |\mathbb{T}_n|)^{1/(n-1)}.\end{equation}

In particular, if $E(\mathbb{T}_1,\cdots,\mathbb{T}_n)$ is the set of points in $\R^n$ each lying in at least one tube of the family $\mathbb{T}_i$ for all $i=1,...,n$, \eqref{eq:guthstrong} implies that
\begin{equation} \label{eq:guthsimple} vol_n(E(\mathbb{T}_1,...,\mathbb{T}_n)) \lesssim_n (|\mathbb{T}_1| \cdots |\mathbb{T}_n|)^{1/(n-1)}.
\end{equation}

It is natural to ask what the analogues of the above would be, if the tubes were shrunk to lines. To formulate such a discrete analogue of the maximal Kakeya operator conjecture, one could ask whether it is possible to have that, for any collection $\mathfrak{L}$ of $L$ lines in $\mathbb{F}^n$, \begin{equation} \label{eq:unknownjoints}\sum_{x \in J(\mathfrak{L})}\#\{\text{lines in }\mathfrak{L}\text{ through }x\}^{\frac{n}{n-1}}\lesssim_n L^{\frac{n}{n-1}}.\end{equation}
Without the genericity hypothesis, this inequality clearly fails in finite-field settings, but we believe it may be true when $\mathbb{F}$ has characteristic 0. Let us mention here that \eqref{eq:unknownjoints} has been recently proved for generic joints in any field setting by Hablicsek (in \cite{Hablicsek_14}). Moreover, we have shown \eqref{eq:unknownjoints} in $\R^3$ in \cite{Iliopoulou_12}, without the genericity hypothesis, but with an extra logarithmic factor of $L$ in its right-hand side, which we believe is not necessary (the logarithmic factor of $1/\delta$ on the right-hand side of \eqref{eq:maximalkakeya} reflects the fact that two tubes may intersect a lot; two lines, on the other hand, intersect at at most one point).

However, Theorem \ref{genericjoints} which we prove in this paper implies that, in $\R^n$, the genericity hypothesis ensures much better bounds than \eqref{eq:unknownjoints}: 

\begin{corollary} Let $\mathfrak{L}$ be a finite collection of $L$ lines in $\R^n$, $n \geq 2$. For each $x \in J(\mathfrak{L})$, we denote by $l(x)$ the number of lines of $\mathfrak{L}$ passing through $x$. Then,
\begin{equation}\label{eq:knowngeneric} \sum_{\{x\in J(\mathfrak{L}):\;x\text{ generic and }l(x)\lesssim_n L^{1/2}\}}l(x)^q\lesssim_{n,q} L^{\frac{n}{n-1}}\text{, for each }0\leq q < \frac{n+1}{n-1},
\end{equation}
and
\begin{equation}\label{eq:less}\sum_{\{x\in J(\mathfrak{L}):\;x\text{ generic and }l(x)\gtrsim_n L^{1/2}\}}l(x)^q\lesssim_{n,q} L^q\text{, for each }q >1.
\end{equation}
\end{corollary}

In other words, \eqref{eq:knowngeneric} holds when the exponent $q$ on the left-hand side is considerably larger than the exponent on the right-hand side \footnote{It might be interesting to investigate to what extent inequalities such as \eqref{eq:knowngeneric} in $\R^n$, where the exponent on the left-hand side is considerably larger than $\frac{n}{n-1}$, relate to Wolff's heuristic in \cite{MR1660476} that joints estimates imply lower bounds on the Minkowski dimension of Kakeya sets; in the heuristic, an estimate with exponent $\frac{n}{n-1}$ is being used, which arises via a probabilistic argument similar to the one giving \eqref{eq:genericprobabilistic}.}, while in \eqref{eq:less} the exponents on both sides are equal but considerably lower than $\frac{n}{n-1}$. Note that, as we have explained earlier, Hablicsek has proved \eqref{eq:knowngeneric} in any field setting, however for $0\leq q\leq \frac{n}{n-1}$, and \eqref{eq:less} in any field setting, but for $q=\frac{n}{n-1}$.

Let us finally mention that it is conjectured by Carbery that, for any field $\mathbb{F}$ and any $n \geq 3$,
$$\sum_{x \in J(\mathfrak{L})} N(x)^{\frac{1}{n-1}} \lesssim_n L^{\frac{n}{n-1}},
$$
where $N(x)$ denotes the number of linearly independent $n$-tuples of lines of $\mathfrak{L}$ through $x$. Note that $N(x)\lesssim_n \#\{$lines of $\mathfrak{L}$ through $x\}^n$ (with equality up to multiplication with constants only for generic joints), and thus Theorem \ref{genericjoints} implies that the statement of the conjecture is true in the generic case in $\R^n$.

As for our multijoints estimate in $\R^n$, it is clearly a discrete analogue of \eqref{eq:guthsimple}.\\ \\
\textbf{Remark.} It is also conjectured by Carbery that, for any transversal collections $\mathfrak{L}_1, \ldots, \mathfrak{L}_n$ of $L_1, \ldots, L_n$ lines in $\mathbb{F}^n$ (transversal in the sense that, whenever $n$ lines, one from each collection, meet at a point, then they form a joint there),
 then 
\begin{equation} \label{eq:multwithmultiplicities}\sum_{x\in  J(\mathfrak{L}_1,\ldots,\mathfrak{L}_n)}\left(\#\{\text{ lines of }\mathfrak{L}_1\text{ through } x\}\cdots\#\{\text{ lines of }\mathfrak{L}_n\text{ through } x\}\right)^{\frac{1}{n-1}} \lesssim_n
\end{equation}
$$\lesssim_n (L_1\cdots L_n)^{\frac{1}{n-1}},
$$ 
an inequality which is the discrete analogue of Guth's endpoint multilinear Kakeya theorem. In \cite{Iliopoulou_13}, we have essentially shown \eqref{eq:multwithmultiplicities} in $\R^3$. More precisely, we have proved that, if $\mathfrak{L}_1$, $\mathfrak{L}_2$, $\mathfrak{L}_3$ are transversal, finite collections of $L_1$, $L_2$ and $L_3$, respectively, lines in $\R^3$, and, for each $x \in J(\mathfrak{L}_1,\mathfrak{L}_2, \mathfrak{L}_3)$ and $i=1,2,3$, $N_i(x)$ denotes the number of lines of $\mathfrak{L}_i$ passing through $x$, then
$$\sum_{\{x \in J(\mathfrak{L}_1,\mathfrak{L}_2,\mathfrak{L}_3):\;N_m(x)>10^{12}\}} \big(N_1(x)N_2(x)N_3(x)\big)^{1/2} \lesssim (L_1L_2L_3)^{1/2},
$$
where $m \in \{1,2,3\}$ is such that $L_m=\min\{L_1,L_2,L_3\}$. We have not yet managed, however, to show that$$\sum_{x \in J(\mathfrak{L}_1,\mathfrak{L}_2,\mathfrak{L}_3)} \big(N_1(x)N_2(x)N_3(x)\big)^{1/2} \lesssim (L_1L_2L_3)^{1/2},
$$
so we will not focus on this result here; details can be found in \cite{Iliopoulou_13}. Let us just mention that the proof of the particular result uses facts from computational geometry that hold only in $\R^3$, which is why we have not managed to apply it to higher dimensions or different field settings.

$\;\;\;\;\;\;\;\;\;\;\;\;\;\;\;\;\;\;\;\;\;\;\;\;\;\;\;\;\;\;\;\;\;\;\;\;\;\;\;\;\;\;\;\;\;\;\;\;\;\;\;\;\;\;\;\;\;\;\;\;\;\;\;\;\;\;\;\;\;\;\;\;\;\;\;\;\;\;\;\;\;\;\;\;\;\;\;\;\;\;\;\;\;\;\;\;\;\;\;\;\;\;\;\;\;\;\;\;\;\;\;\;\;\blacksquare$ 

\textbf{Acknowledgements.} I would like to thank Anthony Carbery, for introducing me to the questions dealt with in this paper, and for his support. I am also very grateful to M\'arton Hablicsek, for communicating to me his work and thoughts. Part of this research was performed while visiting the Institute for Pure and Applied Mathematics (IPAM), which is supported by the National Science Foundation. This reasearch was supported by European Research Council Grant 307617. I would also like to thank the anonymous referee, for their very careful reading of the paper, and for finding certain mistakes in the original version.

\section{First steps and preliminaries}

As we have already mentioned, the basis for our proofs will be the following lemma, which is a generalisation, in $\mathbb{F}^n$, of Quilodr\'an's main idea in \cite{MR2594983} for the solution of the joints problem in $\R^n$.

\begin{lemma} \label{firstpolyd} Let $\mathbb{F}$ be a field, $n \geq 2$. Let $J$ be a set of joints formed by a collection $\mathfrak{L}$ of lines in $\mathbb{F}^n$. If $J$ lies in the zero set of some non-zero polynomial in $\mathbb{F}[x_1,\ldots,x_n]$, of degree at most $d$, then there exists a line in $\mathfrak{L}$ containing at most $d$ points of $J$.
\end{lemma}

The proof of Lemma \ref{firstpolyd} makes use of the notion of the \textit{formal} (or \textit{Hasse}) \textit{derivatives} of a polynomial. More precisely, the \textit{formal} (or \textit{Hasse}) \textit{gradient} of a polynomial $p \in \mathbb{F}[x_1,\ldots, x_n]$ is the element $\nabla{p}:=(p_1,\ldots,p_n)$ of $(\mathbb{F}[x_1,\ldots,x_n])^n$, where, for $i=1,\ldots,n$, the \textit{$i$-th formal derivative} $p_i$ of $p$ is the coefficient of $z_i$ in $p(x+z)$ as a polynomial in $\mathbb{F}[x_1,\ldots,x_n,z_1,\ldots,\hat{z_i},\ldots,z_n]$ (e.g., see \cite{MR2648400}). Note that, when $\mathbb{F}=\R$, the formal and the usual partial derivatives of a polynomial agree. Moreover, in analogy to the situation in $\mathbb{C}$, it is easy to see that, for any algebraically closed field $\mathbb{F}$, if $p \in \mathbb{F}[x_1,\ldots,x_n]$ satisfies $\nabla{p}=0$, then $p=g^{char(\mathbb{F})}$, for some $g \in \mathbb{F}[x_1,\ldots,x_n]$ (e.g., see \cite{Carbery_Iliopoulou_14} for full details).\\ \\
\textit{Proof of Lemma \ref{firstpolyd}.} We can assume, without loss of generality, that $\mathbb{F}$ is algebraically closed. Indeed, if $\overline{\mathbb{F}}$ denotes the algebraic closure of $\mathbb{F}$, we can extend each line in $\mathfrak{L}$ to a line in $\overline{\mathbb{F}}^n$ in the obvious way, and thus form a set of lines $\overline{\mathfrak{L}}$ in $\overline{\mathbb{F}}^n$; then, each joint formed by $\mathfrak{L}$ in $\mathbb{F}^n$ is a joint formed by $\overline{\mathfrak{L}}$ in $\overline{\mathbb{F}}^n$, since a set of $n$ vectors that is linearly independent in $\mathbb{F}^n$ is linearly independent in $\overline{\mathbb{F}}^n$. 

Now, let $p \in \mathbb{F}[x_1, \dots, x_n]$ be a polynomial of minimal degree vanishing on $J$. By assumption, $\deg p \leq d$.

Suppose that each line $l \in \mathfrak{L}$ contains more than $d$ elements of $J$. Then, $p$ vanishes on each line in $\mathfrak{L}$, and therefore the elements of $J$, being joints formed by $\mathfrak{L}$, are singular points of the zero set of $p$. Therefore, $\nabla{p}$ vanishes on $J$, and thus $\nabla{p}=0$, otherwise the minimality of the degree of $p$ would be contradicted. So, $p$ is equal to $g^{char(\mathbb{F})}$, for some $g \in \mathbb{F}[x_1,\ldots, x_n]$. This contradicts the minimality of the degree of $p$ if $char(\mathbb{F})\neq 0$, while, if $char(\mathbb{F})=0$, it implies that $p$ is a constant polynomial, thus the zero polynomial (for $J \neq \emptyset$), which is a contradiction.

\qed

Using Lemma \ref{firstpolyd} in $\mathbb{F}^n$ in exactly the same way as Quilodr\'an uses it in $\R^n$ in \cite{MR2594983}, we deduce the following joints estimate, and the subsequent immediate multijoints estimate.

\begin{lemma} \label{quilodran} Let $\mathfrak{L}$ be a finite collection of lines in $\mathbb{F}^n$, $n \geq 2$, where $\mathbb{F}$ is an arbitrary field. If $J$ is a subset of $J(\mathfrak{L})$, such that there exists a non-zero polynomial in $\mathbb{F}[x_1,\dots, x_n]$, of degree at most $d$, vanishing on $J$, then $$|J| \leq |\mathfrak{L}| \cdot d. $$

\end{lemma}

\begin{proof} By Lemma \ref{firstpolyd}, there exists a line $l_1 \in \mathfrak{L}$ containing at most $d$ elements of $J$. We remove $l_1$ and the elements of $l_1 \cap J$ from our collection of lines and joints, respectively. The elements of $J_1:=J\setminus (l_1\cap J)$ are joints formed by $\mathfrak{L}\setminus \{l_1\}$. Now, there exists a non-zero polynomial in $\mathbb{F}[x_1,\dots, x_n]$, of degree at most $d$, vanishing on $J_1$ (since the same holds for the set $J$). Again by Lemma \ref{firstpolyd}, there exists $l_2 \in \mathfrak{L}\setminus \{l_1\}$, containing at most $d$ elements of $J_1$. We remove $l_2$ and the elements of $l_2\cap J_1$ from our collection of lines and joints, respectively.

We continue as above, until we have removed enough lines of $\mathfrak{L}$ (and the corresponding elements of $J$ on each) so that no elements of $J$ are remaining. This is achieved in at most $|\mathfrak{L}|$ steps, and $J$ is the union of its subsets that were removed in each step, each of which has size at most $d$. Therefore, $|J| \leq |\mathfrak{L}| \cdot d$. 

\end{proof}

\begin{corollary} \label{quilodransimple} Let $\mathfrak{L}_1,\ldots,\mathfrak{L}_n$ be finite collections of lines in $\mathbb{F}^n$, $n \geq 2$, where $\mathbb{F}$ is an arbitrary field. If $J$ is a subset of $J(\mathfrak{L}_1,\ldots,\mathfrak{L}_n)$, such that there exists a non-zero polynomial in $\mathbb{F}[x_1,\dots, x_n]$, of degree at most $d$, vanishing on $J$, then $$|J| \leq |\mathfrak{L}_1\cup\ldots\cup\mathfrak{L}_n| \cdot d. $$

\end{corollary}

In order to obtain an estimate on generic joints in $\R^n$, we will use the following variant of Lemma \ref{quilodran}, which again follows from Lemma \ref{firstpolyd}, and the idea behind which can be found in the proof of \cite[Proposition 4.1]{Carbery_Iliopoulou_14}.

\begin{lemma} \label{quilodranmult} Let $\mathfrak{L}$ be a finite collection of lines in $\mathbb{F}^n$, $n \geq 2$, where $\mathbb{F}$ is an arbitrary field. If $J$ is a subset of $J^k(\mathfrak{L})$, for $k \geq n$, such that there exists a non-zero polynomial in $\mathbb{F}[x_1,\ldots,x_n]$, of degree at most $d$, vanishing on $J$, then $$|J| k \leq n |\mathfrak{L}| d.$$
\end{lemma}

\begin{proof} By Lemma \ref{firstpolyd}, there exists a line $l_1 \in \mathfrak{L}$ containing at most $d$ elements of $J$. We remove $l_1$ from our collection of lines. We also remove from our collection of joints the elements of $l_1 \cap J$ that are not joints for $\mathfrak{L}_1:=\mathfrak{L}\setminus \{l_1\}$ (note that, for $k \gneq n$, no elements of $J$ will be removed at this step, since the elements of $J$ are generic joints for $\mathfrak{L}$, and thus generic joints for $\mathfrak{L}_1$). The set $J_1$ of remaining elements of $J$ is a set of joints formed by the lines in $\mathfrak{L}_1$ (due to the genericity hypothesis). Again by Lemma \ref{firstpolyd}, there exists $l_2 \in \mathfrak{L}_1$, containing at most $d$ elements of $J_1$. We remove $l_2$ from our collection of lines; we also remove from our collection $J_1$ of joints the elements of $l_2\cap J_1$ that are not joints for $\mathfrak{L}_2:= \mathfrak{L}_1 \setminus \{l_2\}$. Let $J_2$ be the set of remaining elements of $J_1$.

We continue as above, until we have removed enough lines of $\mathfrak{L}$ so that no elements of $J$ are joints for the remaining set of lines. This is achieved in at most $|\mathfrak{L}|$ steps. Moreover, for each $x \in J$, there exist at least $k-(n-1)$ lines of $\mathfrak{L}$ passing through $x$ that are removed during the process, as otherwise after the final step of the process there would exist at least $n$ remaining lines of $\mathfrak{L}$ passing through $x$, and thus, due to the genericity hypothesis, $x$ would be a joint for the remaining collection of lines, which is a contradiction. Therefore, if, for each $i$, $\mathfrak{L}_i$ is the set of lines in $\mathfrak{L}$ remaining after the $i$-th step of the process, $J_i$ the set of joints in $J$ remaining after the $i$-th step of the process (where $\mathfrak{L}_0:=\mathfrak{L}$, $J_0:=J$), and $l_i$ is the element of $\mathfrak{L}_{i-1}\setminus \mathfrak{L}_i$, it holds that
$$|\mathfrak{L}|d\geq \sum_{i}|J_{i-1}\cap l_i|\geq \sum_{x \in J}(k-(n-1))=|J|(k-n+1).$$
Therefore, since $k-n+1 \geq k/n$, it follows that
$$|J|k\leq n|\mathfrak{L}|d.$$
\end{proof} 

It is clear by Corollary \ref{quilodransimple} and Lemma \ref{quilodranmult} that the existence of polynomials of low degree, vanishing on a large proportion of our sets of multijoints or generic joints, would imply multijoints and generic joints estimates. The problem therefore lies in finding such polynomials of sufficiently low degrees. As we have already mentioned, in the case of multijoints in $\mathbb{F}^3$ this will be achieved via a probabilistic polynomial degree reduction argument. In the argument, we will use the following lemma, which appears, for example, in \cite{Guth_Katz_2010} and \cite{Guth_14}.

\begin{lemma} \label{guth} Let $\mathfrak{L}$ be a finite collection of $L$ lines in $\mathbb{F}^n$, $n \geq 2$, where $\mathbb{F}$ is an arbitrary field. Then, there exists a non-zero polynomial in $\mathbb{F}[x_1,\ldots,x_n]$, of degree $\lesssim_n L^{\frac{1}{n-1}}$, vanishing on each line of $\mathfrak{L}$.

\end{lemma}

In the case of multijoints or generic joints in $\R^n$, we will find polynomials of low degree vanishing on our sets of multijoints or generic joints using Guth--Katz partitioning:

\begin{theorem}\emph{\textbf{(Guth, Katz, \cite[Theorem 4.1]{Guth_Katz_2010})}}\label{2.1.3} Let $\mathfrak{G}$ be a finite set of $S$ points in $\R^n$, and $d>1$. Then, there exists a non-zero polynomial $p \in \R[x_1,...,x_n]$, of degree $\leq d$, and $\lesssim_n d^n$ pairwise disjoint open sets (cells) $C_1,\ldots,C_m$, each of which contains $\lesssim_n S/d^n$ points of $\mathfrak{G}$, such that $\R^n=C_1\sqcup\ldots\sqcup C_m\sqcup Z_p$, where $Z_p$ is the zero set of $p$.\end{theorem}

More precisely, we will show using Guth--Katz partitioning in $\R^n$ that, if our desired estimates on multijoints (or generic joints) hold in $\R^{n-1}$, then there exists a polynomial of sufficiently low degree, vanishing on a large proportion of our set of multijoints (or generic joints). The induction on $n$ closes because clearly our desired multijoints estimate holds in $\R^2$, while our generic joints estimate also holds in $\R^2$, due to the Szemer\'edi--Trotter theorem:

\begin{theorem} \emph{\textbf{(Szemer\'edi, Trotter)}} Let $\mathcal{P}$ be a finite set of points in $\R^2$ and $\mathfrak{L}$ a finite set of lines in $\R^2$. Then, if $I(\mathcal{P},\mathfrak{L})$ denotes the number of incidences between $\mathcal{P}$ and $\mathfrak{L}$, it holds that
$$I(\mathcal{P},\mathfrak{L})\lesssim |\mathcal{P}|^{2/3}|\mathfrak{L}|^{2/3}+|\mathfrak{L}|+|\mathcal{P}|.
$$
In particular, for any $k \geq 2$, if $\;\mathcal{P}^k$ denotes the set of points in $\mathcal{P}$ each lying in at least $k$ and fewer than $2k$ lines of $\mathfrak{L}$, then
$$|\mathcal{P}^k|\lesssim \frac{|\mathfrak{L}|^2}{k^3}+\frac{|\mathfrak{L}|}{k}.
$$

\end{theorem}

\section{Finding appropriate polynomials}

\subsection{\textbf{Multijoints in }$\boldsymbol{\mathbb{F}^3}$\textbf{.}} We wish to prove:\\ \\
\textbf{Theorem \ref{multijointsarbitrary}} \textit{Let $\mathfrak{L}_1$, $\mathfrak{L}_2$, $\mathfrak{L}_3$ be finite collections of $L_1$, $L_2$ and $L_3$, respectively, lines in $\mathbb{F}^3$, where $\mathbb{F}$ is an arbitrary field. Then,}
$$|J(\mathfrak{L}_1,\mathfrak{L}_2,\mathfrak{L}_3)| \leq c\; (L_1L_2L_3)^{1/2},$$
\textit{where $c$ is an absolute constant.} \\

Theorem \ref{multijointsarbitrary} will be proved combining Corollary \ref{quilodransimple} with Lemma \ref{inductivemultarbitrary} that follows, which ensures the existence of a polynomial of low degree vanishing on a large proportion of our set of multijoints, under the assumption that the set of multijoints is large. Lemma \ref{inductivemultarbitrary} will be proved via a probabilistic polynomial degree reduction argument.

\begin{lemma} \label{inductivemultarbitrary} Let $\mathfrak{L}_1$, $\mathfrak{L}_2$, $\mathfrak{L}_3$ be finite collections of $L_1$, $L_2$ and $L_3$, respectively, lines in $\mathbb{F}^3$, where $L_1$, $L_2 \leq L_3$ and $\mathbb{F}$ is an arbitrary field. Assume that Theorem \ref{multijointsarbitrary} holds for any collections $\mathfrak{L}_1'$, $\mathfrak{L}_2'$, $\mathfrak{L}_3'$ of lines in $\mathbb{F}^3$, such that either $|\mathfrak{L}_1'| \leq L_1$ and $|\mathfrak{L}_2'|\lneq L_2$, or $|\mathfrak{L}_1'| \lneq L_1$ and $|\mathfrak{L}_2'|\leq L_2$. Then, there exists a constant $C>1$,  independent of $\mathfrak{L}_1$, $\mathfrak{L}_2$, $\mathfrak{L}_3$, such that, if $|J(\mathfrak{L}_1,\mathfrak{L}_2,\mathfrak{L}_3)| \geq C\; (L_1L_2L_3)^{1/2}$, there exists a subset $J$ of $J(\mathfrak{L}_1,\mathfrak{L}_2,\mathfrak{L}_3)$, with 
$$|J| \gtrsim |J(\mathfrak{L}_1,\mathfrak{L}_2,\mathfrak{L}_3)|,$$
and a non-zero polynomial $p \in \mathbb{F}[x_1,x_2,x_3]$, with
$$\deg p \lesssim \frac{L_1L_2}{|J(\mathfrak{L}_1,\mathfrak{L}_2,\mathfrak{L}_3)|},$$
vanishing on $J$.

\end{lemma}

\begin{proof} Let us assume that $|J(\mathfrak{L}_1,\mathfrak{L}_2,\mathfrak{L}_3)| \geq C\; (L_1L_2L_3)^{1/2}$, for some constant $C>1$. We will show that, if $C$ is sufficiently large (and independent of $\mathfrak{L}_1$, $\mathfrak{L}_2$, $\mathfrak{L}_3$), there exists $J\subset J(\mathfrak{L}_1,\mathfrak{L}_2,\mathfrak{L}_3)$ with the desired properties.

For each $x \in J(\mathfrak{L}_1,\mathfrak{L}_2,\mathfrak{L}_3)$, we fix a line $l_2(x)\in \mathfrak{L}_2$ passing through $x$, with the property that there exist lines $l_1\in \mathfrak{L}_1$ and $l_3\in \mathfrak{L}_3$ through $x$, such that the directions of $l_1$, $l_2(x)$ and $l_3$ span $\R^3$. We say that $x$ ``chooses" $l_2(x)$; note that each $x \in J(\mathfrak{L}_1,\mathfrak{L}_2,\mathfrak{L}_3)$ chooses only one line.

Let $\mathfrak{L}_2'$ be the set of lines in $\mathfrak{L}_2$, each of which is chosen as $l_2(x)$ by at least $\frac{1}{100}\frac{|J(\mathfrak{L}_1,\mathfrak{L}_2,\mathfrak{L}_3)|}{L_2}$ multijoints $x \in J(\mathfrak{L}_1,\mathfrak{L}_2,\mathfrak{L}_3)$. 

Due to our hypothesis that our set of multijoints is large, we have that $|\mathfrak{L}_2'|\gtrsim L_2$. Indeed, let $\mathfrak{G}$ be the set of points $x \in J(\mathfrak{L}_1,\mathfrak{L}_2,\mathfrak{L}_3)$, each of which chooses as $l_2(x)$ a line in $\mathfrak{L}_2'$. By the definition of $\mathfrak{L}_2'$, there exist fewer than $\frac{|J(\mathfrak{L}_1,\mathfrak{L}_2,\mathfrak{L}_3)|}{100 \; L_2}\;|\mathfrak{L}_2\setminus \mathfrak{L}_2'|\leq \frac{1}{100}|J(\mathfrak{L}_1,\mathfrak{L}_2,\mathfrak{L}_3)|$ points $x \in J(\mathfrak{L}_1,\mathfrak{L}_2,\mathfrak{L}_3)$ choosing a line in $\mathfrak{L}_2\setminus\mathfrak{L}_2'$ as $l_2(x)$ . Thus, $|\mathfrak{G}|\gtrsim |J(\mathfrak{L}_1,\mathfrak{L}_2,\mathfrak{L}_3)|$. Now, each element of $\mathfrak{G}$ is a multijoint for the collections $\mathfrak{L}_1$, $\mathfrak{L}_2'$ and $\mathfrak{L}_3$ of lines. Therefore, if $|\mathfrak{L}_2'| \lneq L_2$, it follows, by our hypotheses, that Theorem \ref{multijointsarbitrary} holds for multijoints formed by $\mathfrak{L}_1$, $\mathfrak{L}_2'$ and $\mathfrak{L}_3$, and thus
$$|J(\mathfrak{L}_1,\mathfrak{L}_2,\mathfrak{L}_3)| \leq  \tilde{c}\; (L_1|\mathfrak{L}_2'|L_3)^{1/2},
$$
for some constant $\tilde{c}$ depending only on the constant $c$ appearing in the statement of Theorem \ref{multijointsarbitrary}. This implies that $|\mathfrak{L}_2'| > \frac{L_2}{4\tilde{c}^2}$, because if it was true that $|\mathfrak{L}_2'| \leq \frac{L_2}{4\tilde{c}^2}$, it would follow that
$$|J(\mathfrak{L}_1,\mathfrak{L}_2,\mathfrak{L}_3)| \leq  \frac{1}{2} (L_1L_2L_3)^{1/2},
$$
which contradicts our hypothesis that our set of multijoints is large. Therefore, it indeed holds that $|\mathfrak{L}_2'| \gtrsim L_2$.

We will now first find, via a probabilistic argument, a small subset $\widetilde{\mathfrak{L}_1}$ of $\mathfrak{L}_1$, with the property that each line in a large proportion $\mathfrak{L}_2''$ of $\mathfrak{L}_2'$ intersects a lot of lines of $\widetilde{\mathfrak{L}_1}$. Then, we will find a polynomial of low degree, vanishing on each line of $\widetilde{\mathfrak{L}_1}$. Since each line in $\mathfrak{L}_2''$ intersects a lot of lines of $\widetilde{\mathfrak{L}_1}$, it will contain a lot of points where the polynomial vanishes, and thus the polynomial will vanish on each line in $\mathfrak{L}_2''$. Then, using the fact that $\mathfrak{L}_2'$ is large, it will follow that $\mathfrak{L}_2''$ is large, therefore the polynomial will vanish on a large proportion of our set of multijoints.

Let us now go into the details. We create a random subset $\widetilde{\mathfrak{L}_1}$ of $\mathfrak{L}_1$, by choosing each line of $\mathfrak{L}_1$ with probability $\frac{d^2}{L_1}$, where $d =C \frac{L_1L_2}{|J(\mathfrak{L}_1,\mathfrak{L}_2,\mathfrak{L}_3)|}$. Note that $d>1$ (since $C>1$ and clearly $|J(\mathfrak{L}_1,\mathfrak{L}_2,\mathfrak{L}_3)|\leq L_1L_2$). Also, since $C$ is such that $|J(\mathfrak{L}_1,\mathfrak{L}_2,\mathfrak{L}_3)|\geq C (L_1L_2L_3)^{1/2}$, we have that $\frac{d^2}{L_1} \leq 1$. Now, (i) and (ii) below hold (because the random variables in question follow explicit binomial distributions):

(i) With probability $\gtrsim 1$, $|\widetilde{\mathfrak{L}_1}| \lesssim d^2$.

(ii) With probability $\gtrsim 1$, there exist $\gtrsim |\mathfrak{L}_2'|$ lines of $\mathfrak{L}_2'$, each containing $\geq \frac{C}{10^{12}}d$ points belonging to lines of $\widetilde{\mathfrak{L}_1}$. Indeed, each $x \in J(\mathfrak{L}_1,\mathfrak{L}_2,\mathfrak{L}_3)$ belongs to a line of a random set $\widetilde{\mathfrak{L}_1}$ as above with probability $\geq \frac{d^2}{L_1}$, since there exists a line of $\mathfrak{L}_1$ through $x$. Therefore, for each $l_2 \in \mathfrak{L}_2'$, the expected number of points of $J(\mathfrak{L}_1,\mathfrak{L}_2,\mathfrak{L}_3)$ on $l_2$ that belong to lines of $\widetilde{\mathfrak{L}_1}$ is $\geq \frac{d^2}{L_1} \cdot \frac{1}{100}\frac{|J(\mathfrak{L}_1,\mathfrak{L}_2,\mathfrak{L}_3)|}{L_2}$; therefore, with probability $\gtrsim 1$, at least $ \frac{1}{10^{10}}\frac{d^2}{L_1} \cdot \frac{1}{100}\frac{|J(\mathfrak{L}_1,\mathfrak{L}_2,\mathfrak{L}_3)|}{L_2}=\frac{C}{10^{12}}d$ points of $J(\mathfrak{L}_1,\mathfrak{L}_2,\mathfrak{L}_3)$ on $l_2$ belong to lines of $\widetilde{\mathfrak{L}_1}$. So, with probability $\gtrsim 1$, there exist $\gtrsim |\mathfrak{L}_2'|$ lines of $\mathfrak{L}_2'$, each containing $\geq \frac{C}{10^{12}}d$ points of $J(\mathfrak{L}_1,\mathfrak{L}_2,\mathfrak{L}_3)$ belonging to lines of $\widetilde{\mathfrak{L}_1}$.

Note that, in (i), the probability that $|\widetilde{\mathfrak{L}_1}| \lesssim d^2$ can be as close to 1 as we wish (and independent of $\mathfrak{L}_1$, $\mathfrak{L}_2$, $\mathfrak{L}_3$), if the implicit constant in the inequality $|\widetilde{\mathfrak{L}_1}| \lesssim d^2$ is accordingly large. Therefore, it follows by (i) and (ii) that there exists a subset $\widetilde{\mathfrak{L}_1}$ of $\mathfrak{L}_1$, with $|\widetilde{\mathfrak{L}_1}| \leq C'd^2$ (for some absolute constant $C'$), and a subset $\mathfrak{L}_2''$ of $\mathfrak{L}_2'$, with $|\mathfrak{L}_2''| \gtrsim |\mathfrak{L}_2'|$, such that each line of $\mathfrak{L}_2''$ contains $\geq \frac{C}{10^{12}}d$ points of $J(\mathfrak{L}_1,\mathfrak{L}_2,\mathfrak{L}_3)$ belonging to lines of $\widetilde{\mathfrak{L}_1}$.

By Lemma \ref{guth}, there exists a non-zero polynomial $p \in \mathbb{F}[x_1,x_2,x_3]$, of degree $\leq K |\widetilde{\mathfrak{L}_1}|^{1/2} \leq K(C'd^2)^{1/2}=KC'^{1/2}d$ (where $K$ is an absolute constant), vanishing on each line of $\widetilde{\mathfrak{L}_1}$. In particular, if $\frac{C}{10^{12}}d \gneq KC'^{1/2}d$ (i.e. $C\gneq 10^{12} KC'^{1/2}$), each line of $\mathfrak{L}_2''$ contains $\gneq KC'^{1/2}d\geq \deg p$ points of the zero set of $p$. So, fixing $C$ to be any absolute constant $\gneq 10^{12} KC'^{1/2}$, we have that $p$ vanishes on each line of $\mathfrak{L}_2''$. However, $\mathfrak{L}_2''\subset \mathfrak{L}_2'$; thus, each line in $\mathfrak{L}_2''$ is chosen as $l_2(x)$ by $\gtrsim \frac{|J(\mathfrak{L}_1,\mathfrak{L}_2,\mathfrak{L}_3)|}{L_2}$ points $x \in J(\mathfrak{L}_1,\mathfrak{L}_2,\mathfrak{L}_3)$ (each of which chooses as $l_2(x)$ only one line in $\mathfrak{L}_2''$). Therefore, $p$ vanishes on $\gtrsim \frac{|J(\mathfrak{L}_1,\mathfrak{L}_2,\mathfrak{L}_3)|}{L_2}|\mathfrak{L}_2''| \sim \frac{|J(\mathfrak{L}_1,\mathfrak{L}_2,\mathfrak{L}_3)|}{L_2}|\mathfrak{L}_2'|$ points of $J(\mathfrak{L}_1,\mathfrak{L}_2,\mathfrak{L}_3)$. Since $|\mathfrak{L}_2'|\gtrsim L_2$, the proof of the lemma is complete. 

\end{proof}

\textit{Proof of Theorem \ref{multijointsarbitrary}.} Theorem \ref{multijointsarbitrary} will be proved by induction on the cardinalities of $\mathfrak{L}_1$ and $\mathfrak{L}_2$, combining Lemma \ref{inductivemultarbitrary} and Corollary \ref{quilodransimple}. 

More particularly, Theorem \ref{multijointsarbitrary} holds for any collections $\mathfrak{L}_1$, $\mathfrak{L}_2$, $\mathfrak{L}_3$ of lines in $\mathbb{F}^3$ with $|\mathfrak{L}_1|=|\mathfrak{L}_2|=1$, for any constant $c \geq 1$, as $\mathfrak{L}_1$, $\mathfrak{L}_2$ and $\mathfrak{L}_3$ form at most one joint in this case.

Let us assume that, for some $(M_1,M_2) \in (\N^*)^2$ (where $\N^*=\N\setminus \{0\}$), Theorem \ref{multijointsarbitrary} holds, for some constant $c \geq 1$ which will be specified later, for any collections $\mathfrak{L}_1$, $\mathfrak{L}_2$, $\mathfrak{L}_3$ of lines in $\mathbb{F}^3$ with $|\mathfrak{L}_1|, |\mathfrak{L}_2| \leq |\mathfrak{L}_3|$, such that either $|\mathfrak{L}_1| \leq M_1$ and $|\mathfrak{L}_2|\lneq M_2$, or $|\mathfrak{L}_1|\lneq M_1$ and $|\mathfrak{L}_2| \leq M_2$.

We will show that Theorem \ref{multijointsarbitrary} then holds for any collections $\mathfrak{L}_1$, $\mathfrak{L}_2$, $\mathfrak{L}_3$ of lines in $\mathbb{F}^3$ with $|\mathfrak{L}_1|, |\mathfrak{L}_2| \leq |\mathfrak{L}_3|$, $|\mathfrak{L}_1|= M_1$ and $|\mathfrak{L}_2| = M_2$, for the same constant $c$ as in the induction hypothesis. 

Indeed, let $\mathfrak{L}_1$, $\mathfrak{L}_2$, $\mathfrak{L}_3$ be collections of $L_1$, $L_2$ and $L_3$, respectively, lines in $\mathbb{F}^3$, with $L_1,L_2 \leq L_3$, $L_1=M_1$ and $L_2=M_2$. Let us assume that $|J(\mathfrak{L}_1,\mathfrak{L}_2,\mathfrak{L}_3)|$ $\geq C (L_1L_2L_3)^{1/2}$, where $C$ is the constant appearing in the statement of Lemma \ref{inductivemultarbitrary}. By Lemma \ref{inductivemultarbitrary}, there exists a subset $J$ of $J(\mathfrak{L}_1,\mathfrak{L}_2,\mathfrak{L}_3)$, with 
$$|J| \gtrsim |J(\mathfrak{L}_1,\mathfrak{L}_2,\mathfrak{L}_3)|,$$
and a non-zero polynomial $p \in \mathbb{F}[x_1,x_2,x_3]$, with
$$\deg p \lesssim \frac{L_1L_2}{|J(\mathfrak{L}_1,\mathfrak{L}_2,\mathfrak{L}_3)|},$$
vanishing on $J$. Therefore, Corollary \ref{quilodransimple} implies that
$$|J| \leq |\mathfrak{L}_1\cup \mathfrak{L}_2\cup \mathfrak{L}_3|\cdot\deg p\lesssim L_3\cdot \frac{L_1L_2}{|J(\mathfrak{L}_1,\mathfrak{L}_2,\mathfrak{L}_3)|},
$$
which gives
$$|J(\mathfrak{L}_1,\mathfrak{L}_2,\mathfrak{L}_3)| \leq C' (L_1L_2L_3)^{1/2}
$$
after rearranging, for some absolute constant $C'$. By fixing the constant $c$ in the inductive hypothesis to be equal to $\max\{C,C'\}$, we have that, for the same constant $c$,
$$|J(\mathfrak{L}_1,\mathfrak{L}_2,\mathfrak{L}_3)| \leq c\; (L_1L_2L_3)^{1/2}.
$$
\qed

\subsection{\textbf{Multijoints in }$\boldsymbol{\R^n}$\textbf{.}} We wish to prove:\\ \\
\textbf{Theorem \ref{theoremmult2}} \textit{Let $n \geq 2$. Let $\mathfrak{L}_1,\dots, \mathfrak{L}_n$ be finite collections of $L_1, \dots, L_n$, respectively, lines in $\R^n$. Then,}
\begin{equation} \label{eq:proved}|J(\mathfrak{L}_1,\dots, \mathfrak{L}_n)|\leq c_n (L_1\cdots L_n)^{1/(n-1)}, \end{equation}
\textit{where $c_n$ is a constant depending only on $n$.}\\

This will be achieved by combining Corollary \ref{quilodransimple} with Lemma \ref{inductive} that follows, which ensures the existence of a polynomial of low degree vanishing on a large proportion of our set of multijoints in $\R^n$, for $n \geq 3$, under the assumption that Theorem \ref{theoremmult2} holds in $\R^{n-1}$. In particular, given that Theorem \ref{theoremmult2} obviously holds in $\R^2$, the proof of Theorem \ref{theoremmult2} will be completed by induction on $n$. As we have already mentioned, Lemma \ref{inductive} will be proved using Guth--Katz partitioning.

\begin{lemma} \label{inductive} Let $n \geq 3$. Assume that Theorem \ref{theoremmult2} holds in $\R^{n-1}$. Then, for any $\mathfrak{L}_1, \dots \mathfrak{L}_n$ finite collections of $L_1, \ldots, L_n$, respectively, lines in $\R^n$, with $L_1 \leq L_2 \leq \ldots \leq L_n$, there exists a subset $J$ of $J(\mathfrak{L}_1, \ldots ,\mathfrak{L}_n)$, with $$|J| \gtrsim |J(\mathfrak{L}_1, \ldots ,\mathfrak{L}_n)|,$$ and a non-zero polynomial $p \in \R[x_1, \dots, x_n]$, with $$\deg p \lesssim_n \frac{L_1 \cdots L_{n-1}}{|J(\mathfrak{L}_1, \ldots ,\mathfrak{L}_n)|^{n-2}},$$ vanishing on $J$.

\end{lemma}

\begin{proof} We observe that, if we project $\R^n$ on a generic hyperplane, all the elements of the projection of $J(\mathfrak{L}_1, \ldots \mathfrak{L}_n)$ are multijoints in the hyperplane (i.e., in $\R^{n-1}$), formed by the collections $\mathfrak{L}_1', \ldots, \mathfrak{L}_{n-1}'$ of lines, where $\mathfrak{L}_i'$ is the set of projected lines of $\mathfrak{L}_i$ from $\R^n$ to the hyperplane, for all $i=1, \ldots, n-1$. Therefore, by our assumption that Theorem \ref{theoremmult2} holds in $\R^{n-1}$, we have that $$|J(\mathfrak{L}_1, \ldots, \mathfrak{L}_n)| \lesssim_n (L_1\cdots L_{n-1})^{1/(n-2)},$$and thus, for some large constant $A_n$ which will be fixed later, the quantity $$d=A_n\cdot\frac{L_1\cdots L_{n-1}}{|J(\mathfrak{L}_1, \ldots ,\mathfrak{L}_n)|^{n-2}} $$
is larger than 1. We can thus apply Guth--Katz partitioning (Theorem \ref{2.1.3}) in $\R^n$, for the set $J(\mathfrak{L}_1, \ldots ,\mathfrak{L}_n)$ and this $d$ as the upper bound for the degree of the partitioning polynomial. In particular, there exists a non-zero polynomial $p \in \R[x_1,\ldots, x_n]$, of degree at most $d$, whose zero set $Z$ decomposes $\R^n$ in $\lesssim_n d^n$ cells, each containing $\lesssim_n |J(\mathfrak{L}_1, \ldots ,\mathfrak{L}_n)|/d^n$ elements of $J(\mathfrak{L}_1, \ldots, \mathfrak{L}_n)$.

We now prove that, for $A_n$ sufficiently large, there exist $\gtrsim |J(\mathfrak{L}_1, \ldots ,\mathfrak{L}_n)|$ elements of $J(\mathfrak{L}_1, \ldots ,\mathfrak{L}_n)$ on the zero set of $p$. 

Indeed, assume for contradiction that there exist $\gtrsim |J(\mathfrak{L}_1, \ldots ,\mathfrak{L}_n)|$ elements of $J(\mathfrak{L}_1, \ldots \mathfrak{L}_n)$ in the union of the cells. Then, there exist $\gtrsim_n d^n$ cells, each containing $\gtrsim_n |J(\mathfrak{L}_1, \ldots \mathfrak{L}_n)|/d^n$ elements of $J(\mathfrak{L}_1, \ldots ,\mathfrak{L}_n)$. We call these cells ``full cells", and we denote the set of full cells by $\mathcal{C}$.

Now, for each $i=1,...,n$, there exist $\gtrsim_n |\mathcal{C}|$ full cells, each intersected by $\lesssim_n L_i/d^{n-1}$ lines in $\mathfrak{L}_i$ (where the first implicit multiplicative constant is smaller than 1, and the second larger than 1). Indeed, suppose that this is false for some $i \in \{1,\ldots,n\}$. Then, the number of full cells intersected by $\lesssim_n L_i/d^{n-1}$ lines in $\mathfrak{L}_i$ is $\lesssim_n |\mathcal{C}|$, thus there exist $\gtrsim_n |\mathcal{C}| \sim_n d^n$ cells, each intersected by $\gtrsim_n L_i/d^{n-1}$ lines in $\mathfrak{L}_i$. Choosing appropriately the constants hiding behind the $\gtrsim_n$ symbols, it follows that there exist $\gneq d^n \cdot L_i/d^{n-1}= L_i d$ incidences between the zero set $Z$ of $p$ and the lines of $\mathfrak{L}_i$ not in $Z$, which is a contradiction, since each line not lying in $Z$ can intersect $Z$ at most $d$ times. We denote by $\mathcal{C}_i$ the set of full cells intersected by $\lesssim_n L_i/d^{n-1}$ lines in $\mathfrak{L}_i$. We have just shown that $|\mathcal{C}_i| \geq c_{i,n} |\mathcal{C}|$, where $c_{i,n}$ is a constant smaller than 1, depending only on $i$ and $n$; note that we can choose it to be as close to 1 as we wish, of course changing appropriately the rest of the constants hiding behind the $\gtrsim_n$ symbols in the above analysis.

Now, by choosing the constants $c_{1,n}, \ldots, c_{n,n}$ to be appropriately large, it follows that there exists some full cell that is intersected by $\lesssim_n L_i/d^{n-1}$ lines in $\mathfrak{L}_i$, for all $i=1,\ldots ,n$. The lines intersecting the cell form the multijoints in the cell, which are $\gtrsim_n |J(\mathfrak{L}_1, \ldots ,\mathfrak{L}_n)|/d^n$. Therefore, projecting on a generic hyperplane and applying there Theorem \ref{theoremmult2} (which we have assumed holds in $\R^{n-1}$), we see that
$$\frac{|J(\mathfrak{L}_1, \ldots, \mathfrak{L}_n)|}{d^n}\lesssim_n \bigg( \frac{L_1}{d^{n-1}}\cdots \frac{L_{n-1}}{d^{n-1}}\bigg)^{1/(n-2)},$$
which means that 
$$d\lesssim_n \frac{L_1\cdots L_{n-1}}{|J(\mathfrak{L}_1, \ldots, \mathfrak{L}_n)|^{n-2}},$$
a contradiction for $A_n$ sufficiently large. The proof is complete.

\end{proof}

\textit{Proof of Theorem \ref{theoremmult2}.} Theorem \ref{theoremmult2} holds for $n=2$. Indeed, if $\mathfrak{L}_1$, $\mathfrak{L}_2$ are finite collections of $L_1$, $L_2$, respectively, lines in $\R^2$, then there exists at least one line of each collection through each point of $J(\mathfrak{L}_1,\mathfrak{L}_2)$, so $|J(\mathfrak{L}_1,\mathfrak{L}_2)|$ is equal to at most the number of all the pairs of the form $(l_1, l_2)$, where $l_1 \in \mathfrak{L}_1$, $l_2 \in \mathfrak{L}_2$, i.e. 
$$|J(\mathfrak{L}_1,\mathfrak{L}_2)| \leq L_1 L_2. $$

For $n \geq 3$, assume that Theorem \ref{theoremmult2} holds in $\R^{n-1}$.

Let $\mathfrak{L}_1, \dots ,\mathfrak{L}_n$ be finite collections of $L_1, \ldots, L_n$, respectively, lines in $\R^n$, with $L_1 \leq L_2 \leq \ldots \leq L_n$. By Lemma \ref{inductive}, there exists a subset $J$ of $J(\mathfrak{L}_1, \dots ,\mathfrak{L}_n)$, with $$|J| \gtrsim |J(\mathfrak{L}_1, \ldots ,\mathfrak{L}_n)|,$$ and a non-zero polynomial in $p \in \R[x_1, \dots, x_n]$, with $$\deg p \lesssim_n \frac{L_1 \cdots L_{n-1}}{|J(\mathfrak{L}_1, \ldots ,\mathfrak{L}_n)|^{n-2}},$$ vanishing on $J$. By Corollary \ref{quilodransimple}, $|J| \lesssim_n L_n \cdot \deg p$, so $$|J(\mathfrak{L}_1, \dots ,\mathfrak{L}_n)| \lesssim_n L_n \cdot\frac{L_1 \cdots L_{n-1}}{|J(\mathfrak{L}_1, \ldots ,\mathfrak{L}_{n-1})|^{n-2}},$$ and thus $|J(\mathfrak{L}_1, \dots ,\mathfrak{L}_n)| \lesssim_n (L_1\cdots L_n)^{1/(n-1)}$, by rearranging. Therefore, Theorem \ref{theoremmult2} holds in $\R^n$. \newline\qed

\subsection{\textbf{Generic joints.}} We wish to prove:\\ \\
\textbf{Theorem \ref{genericjoints}} \textit{Let $n \geq 2$. Let $\mathfrak{L}$ be a finite collection of L lines in $\R^n$. Then,}
$$|J^{k}(\mathfrak{L})| \leq c_n \left(\frac{L^{\frac{n}{n-1}}}{k^{\frac{n+1}{n-1}}}+\frac{L}{k}\right)$$ \textit{for all $k \geq n$, where $c_n$ is a constant depending only on $n$.}\\

In analogy to the multijoints in $\R^n$ situation, this will be achieved by combining Lemma \ref{quilodranmult} with Lemma \ref{jointsinductive} that follows, which ensures the existence of a polynomial of low degree vanishing on a large proportion of our set of generic joints in $\R^n$, for $n \geq 3$, under the assumptions that our set of joints is large and that Theorem \ref{genericjoints} holds in $\R^{n-1}$. In particular, given that Theorem \ref{genericjoints} holds in $\R^2$ (by the Szemer\'edi--Trotter theorem), the proof of Theorem \ref{genericjoints} will be completed by induction on $n$. As we have already mentioned, Lemma \ref{jointsinductive} will be proved using Guth--Katz partitioning.

\begin{lemma} \label{jointsinductive} Let $n\geq 3$. Assume that Theorem \ref{genericjoints} holds in $\R^{n-1}$, for some constant $c_{n-1}$. Let $\mathfrak{L}$ be a finite collection of $L$ lines in $\R^n$, and $k \geq n$, such that $$|J^k(\mathfrak{L})|\geq C_n \left(\frac{L^{\frac{n}{n-1}}}{k^{\frac{n+1}{n-1}}}+\frac{L}{k}\right)$$ for some sufficiently large constant $C_n$, depending only on $n$. Then, there exists a subset $J^k$ of $J^k(\mathfrak{L})$, with 
$$|J^k| \gtrsim |J^k(\mathfrak{L})|,$$
and a non-zero polynomial $p \in \R[x_1,\ldots,x_n]$, with
$$\deg p \leq A_n \frac{L^{n-1}}{|J^k(\mathfrak{L})|^{n-2}k^n},$$
vanishing on $J^k$, where $A_n$ is a constant depending only on $n$ and $c_{n-1}$.
\end{lemma}

\begin{proof} If we project $\R^n$ on a generic hyperplane, all the elements of the projection of $J^k(\mathfrak{L})$ are generic joints in the hyperplane (i.e., in $\R^{n-1}$) formed by the collection $\mathfrak{L}'$ of lines, where $\mathfrak{L}'$ is the set of projected lines of $\mathfrak{L}$ from $\R^n$ to the hyperplane; in particular, there exist at least $k$ and fewer than $2k$ lines of $\mathfrak{L}'$ through each of these joints. Therefore, by our assumption that Theorem \ref{theoremmult2} holds in $\R^{n-1}$, we have that 
$$|J^k(\mathfrak{L})| \leq c_{n-1}\left(\frac{L^{\frac{n-1}{n-2}}}{k^{\frac{n}{n-2}}}+\frac{L}{k}\right),$$
where, for $n\geq 2$, $c_n$ is the constant appearing in the statement of Theorem \ref{genericjoints}. 

On the other hand, for $C_n \geq 2c_{n-1}$, $|J^k(\mathfrak{L})|\geq 2c_{n-1} \frac{L}{k}$ by assumption, so $\frac{L^{\frac{n-1}{n-2}}}{k^{\frac{n}{n-2}}}\geq \frac{L}{k}$, and 
$$|J^k(\mathfrak{L})| \leq 2c_{n-1} \frac{L^{\frac{n-1}{n-2}}}{k^{\frac{n}{n-2}}}.$$
Thus, assuming from now on that $C_n \geq 2c_{n-1}$, we have that, for some constant $A_n\geq 2c_{n-1}$ (and independent of $C_n$) which will be fixed later, the quantity $$d=A_n\frac{L^{n-1}}{|J^k(\mathfrak{L})|^{n-2}k^n}$$
is larger than 1. We can thus apply Guth--Katz partitioning (Theorem \ref{2.1.3}) in $\R^n$, for the set $J^k(\mathfrak{L})$ and this $d$ as the upper bound for the degree of the partitioning polynomial. In particular, there exists a non-zero polynomial $p \in \R[x_1,\ldots, x_n]$, of degree at most $d$, whose zero set $Z$ decomposes $\R^n$ in $\lesssim_n d^n$ cells, each containing $\lesssim_n |J^k(\mathfrak{L})|/d^n$ elements of $J^k(\mathfrak{L})$.

By fixing $A_n$ to be sufficiently large, there exist $\gtrsim |J^k(\mathfrak{L})|$ elements of $J^k(\mathfrak{L})$ lying in $Z$. Indeed, let us assume for contradiction that there exist $\gtrsim |J^k(\mathfrak{L})|$ elements of $J^k(\mathfrak{L})$ in the union of the cells. Then, there exist $\gtrsim_n d^n$ cells, each containing $\gtrsim_n |J^k(\mathfrak{L})|/d^n$ elements of $J^k(\mathfrak{L})$. We call these cells ``full cells", and we denote the set of full cells by $\mathcal{C}$.

Suppose that there exists a full cell containing $\leq k$ elements of $J^k(\mathfrak{L})$. Then, $|J^k(\mathfrak{L})|/d^n \lesssim_n k$, which implies that
$$|J^k(\mathfrak{L})| \leq C'_n\frac{L^{\frac{n}{n-1}}}{k^{\frac{n+1}{n-1}}},$$
for some constant $C'_n$ depending only on $n$. Therefore, for $C_n \gneq C'_n$, we have a contradiction, and thus there exist $\geq k$ elements in each full cell. We continue assuming that $C_n \gneq C'_n$.

There exist $\gtrsim_n d^n$ full cells, each intersected by $\lesssim_n L/d^{n-1}$ lines in $\mathfrak{L}$ (where the first implicit multiplicative constant is smaller than 1, and the second larger than 1). Indeed, suppose that this is false, and that in fact the number of full cells intersected by $\lesssim_n L/d^{n-1}$ lines in $\mathfrak{L}$ is $\leq c_n |\mathcal{C}|$, for some constant $c_n <1$. Then, there exist $\gtrsim_n |\mathcal{C}| \sim_n d^n$ cells, each intersected by $\gtrsim_n L/d^{n-1}$ lines in $\mathfrak{L}$. Choosing appropriately the constants hiding behind the $\gtrsim_n$ symbols, it follows that there exist $\gneq d^n \cdot L/d^{n-1}= L d$ incidences between the zero set $Z$ of $p$ and the lines of $\mathfrak{L}$ not in $Z$, which is a contradiction, since each line not lying in $Z$ can intersect $Z$ at most $d$ times.

So, there exists some full cell that is intersected by $\lesssim_n L/d^{n-1}$ lines in $\mathfrak{L}$; in other words, if $L_{cell}$ is the number of lines intersecting the full cell, $L_{cell} \lesssim_n L/d^{n-1}$. On the other hand, the lines intersecting the cell form the generic joints in $J^k(\mathfrak{L})$ that lie in the cell, which are $\gtrsim_n |J^k(\mathfrak{L})|/d^n$. Therefore, projecting on a generic hyperplane and applying there Theorem \ref{genericjoints} (which we have assumed holds in $\R^{n-1}$), we see that
$$\frac{|J^k(\mathfrak{L})|}{d^n}\lesssim_n \frac{L_{cell}^{\frac{n-1}{n-2}}}{k^{\frac{n}{n-2}}}+\frac{L_{cell}}{k}  .$$
However, there exist $> k$ generic joints in the cell, each of which lies in at least $k$ lines of $\mathfrak{L}$; thus, $L_{cell} \geq k+(k-1)+\ldots +1 \gtrsim k^2$, from which it follows that
$$\frac{L_{cell}^{\frac{n-1}{n-2}}}{k^{\frac{n}{n-2}}}\gtrsim_n \frac{L_{cell}}{k}.$$
Therefore, $$\frac{|J^k(\mathfrak{L})|}{d^n}\lesssim_n \frac{L_{cell}^{\frac{n-1}{n-2}}}{k^{\frac{n}{n-2}}}\lesssim_n \frac{\bigg( \frac{L}{d^{n-1}}\bigg)^{\frac{n-1}{n-2}}}{k^{\frac{n}{n-2}}},$$
which in turn implies
$$d\leq C''_n \frac{L^{n-1}}{|J^k(\mathfrak{L})|^{n-2}k^n},$$
for some constant $C''_n$ depending only on $n$.
By fixing $A_n$ to be larger than $C''_n$ (which is independent of $C_n$), we are led to a contradiction. This completes the proof.

\end{proof}

\textit{Proof of Theorem \ref{genericjoints}.} Theorem \ref{genericjoints} holds in $\R^2$ (by the Szemer\'edi--Trotter theorem). Assume that Theorem \ref{genericjoints} holds in $\R^{n-1}$, for some $n \geq 3$. We will show that Theorem \ref{genericjoints} holds in $\R^n$.

Let $\mathfrak{L}$ be a finite collection of $L$ lines in $\R^n$ and $k \geq n$, such that
$$|J^k(\mathfrak{L})|\geq C_n \left(\frac{L^{\frac{n}{n-1}}}{k^{\frac{n+1}{n-1}}}+\frac{L}{k}\right),$$
where $C_n$ is the constant appearing in the statement of Lemma \ref{jointsinductive}. It follows by Lemma \ref{jointsinductive} that there exists a subset $J^k$ of $J^k(\mathfrak{L})$ with 
$$|J^k| \gtrsim |J^k(\mathfrak{L})|,$$
and a non-zero polynomial $p \in \R[x_1,\ldots,x_n]$, with
$$\deg p \lesssim_n \frac{L^{n-1}}{|J^k(\mathfrak{L})|^{n-2}k^n},$$
vanishing on $J^k$. Therefore, by Lemma \ref{quilodranmult}, $$|J^k|k\lesssim_n L \cdot \deg p\lesssim_n L\; \frac{L^{n-1}}{|J^k(\mathfrak{L})|^{n-2}k^n},$$
which implies
$$|J^k(\mathfrak{L})|\lesssim_n\frac{L^{\frac{n}{n-1}}}{k^{\frac{n+1}{n-1}}}$$
by rearranging. The proof of the theorem is now complete.

\qed

\section{Transversality of more general curves} \label{section6}

In this section, we generalise Theorems \ref{theoremmult2} and \ref{genericjoints} to multijoints and generic joints formed by real algebraic curves in $\R^n$.

We consider the family $\mathcal{F}_n$ of all subsets $\gamma$ of $\R^n$ with the property that, if $x \in \gamma$, then a basic neighbourhood of $x$ in $\gamma$ is either $\{x\}$ or the finite union of parametrised curves, each homeomorphic to a semi-open line segment with one endpoint the point $x$. If there exists a local parametrisation $f:[0,1)\rightarrow \gamma$ of a $\gamma \in \Gamma$, such that $f(0)=x$ and $f'(0)\neq 0$, then the line in $\R^n$ passing through $x$ with direction $f'(0)$ is tangent to $\gamma$ at $x$. If $\Gamma \subset \mathcal{F}_n$, we denote by $T_x^{\Gamma}$ the set of directions of all tangent lines at $x$ to the elements of $\Gamma$ passing through $x$ (note that $T_x^{\Gamma}$ may be empty and that there may exist many tangent lines to an element of $\Gamma$ at $x$).

\begin{definition} Let $\Gamma$ be a finite collection of sets in $\mathcal{F}_n$. A point $x$ in $\R^n$ is a \emph{joint} formed by $\Gamma$ if there exist $v_1,\ldots,v_n \in T_x^{\Gamma}$ such that $\{v_1,\ldots,v_n\}$ spans $\R^n$.

A point $x \in \R^n$ is a \emph{generic joint} formed by $\Gamma$ if any $n$ vectors in $T_x^{\Gamma}$ span $\R^n$.

For eack $k \geq n$, we denote by $J^k(\Gamma)$ the set of generic joints $x$ formed by $\Gamma$, such that $|T_x^{\Gamma}|\in [k,2k)$.

\end{definition}

\begin{definition} Let $\Gamma_1$, ..., $\Gamma_n$ be finite collections of sets in $\mathcal{F}_n$. A point $x$ in $\R^n$ is a \emph{multijoint} formed by these collections if, for all $i=1,\ldots,n$, $x$ belongs to some $\gamma_i \in \Gamma_i$, with the property that there exists at least one vector $v_i$ in $T_x^{\{\gamma_i\}}$, such that $\{v_1,...,v_n\}$ spans $\R^n$.

We denote by $J(\Gamma_1,...,\Gamma_n)$ the set of multijoints formed by $\Gamma_1$, ..., $\Gamma_n$.
\end{definition}

Under certain assumptions on the properties of the sets in $\mathcal{F}_n$, the corresponding statements of Theorems \ref{theoremmult2} and \ref{genericjoints} still hold:

\begin{theorem} \label{theoremmult3} Let $n \geq 2$ and $b>0$. Let $\Gamma_1, \ldots, \Gamma_n$ be finite collections of irreducible real algebraic curves in $\R^n$, of degree at most $b$. Then,
\begin{displaymath}|J(\Gamma_1, \ldots, \Gamma_n)| \leq c_{n,b}\;(|\Gamma_1|\cdots |\Gamma_n|)^{1/(n-1)}, \end{displaymath}
where $c_{n,b}$ is a constant depending only on $n$ and $b$.
\end{theorem}

\begin{theorem} \label{genericjoints2} Let $n \geq 2$, $b>0$. Let $\Gamma$ be a finite collection of real algebraic curves in $\R^n$, of degree at most $b$. Then, for all $k \geq n$,
$$|J^{k}(\Gamma)| \leq c_{n,b} \left(\frac{|\Gamma|^{\frac{n}{n-1}}}{k^{\frac{n}{n-1}+\frac{1}{n-1}\cdot \frac{1}{B-1}}}+\frac{|\Gamma|}{k}\right),$$ 
where $B:=\binom{b+2}{2}-1$.
\end{theorem}

(Note that Theorem \ref{genericjoints2} implies Theorem \ref{genericjoints}.)

We clarify here that every real algebraic curve $\gamma$ in $\R^n$ is contained in a complex algebraic curve in $\C^n$ (viewing $\R^n$ as a subset of $\C^n$), and the degree of $\gamma$ is defined as the degree of the smallest complex algebraic curve containing $\gamma$ (e.g., see \cite{Iliopoulou_12} for full details).

The proofs of Theorems \ref{theoremmult3} and \ref{genericjoints2} are similar to the proofs of Theorems \ref{theoremmult2} and \ref{genericjoints}, respectively, we will thus give a sketch of them, without dwelling on all details. Crucial for the proofs of both theorems is the following.

\begin{lemma}\label{projection} Let $n \geq 2$ and $\gamma$ a real algebraic curve in $\R^n$. Then, the projection of $\gamma$ on a generic hyperplane is contained in some real algebraic curve in the hyperplane, of degree at most $\deg \gamma$.

\end{lemma}

\begin{proof} Each real algebraic curve in $\R^n$ is contained in a complex algebraic curve in $\C^n$ with the same degree. Moreover, the intersection of any complex algebraic curve in $\C^n$ with $\R^n$ is a real variety that is either 0-dim or 1-dim (see \cite{Iliopoulou_14}). So, since the projection of a real algebraic curve on a generic hyperplane cannot be contained in a 0-dim variety, it suffices to prove the Lemma for projections of complex algebraic curves instead.

Let $\gamma$ be a complex algebraic curve in $\C^n$, of degree $b$. Let $\Pi$ be a generic hyperplane in $\C^n$. By a change of variables, projecting any $x=(x_1,\ldots,x_n)\in \C^n$ on $\Pi$ corresponds to eliminating $x_1$. Thus (see \cite{CLO_91}), the smallest complex variety in $\Pi$ containing the projection $p(\gamma)$ of $\gamma$ on $\Pi$ is $V(I_1(\gamma))$, whose ideal is the elimination ideal $I_1(\gamma)=I(\gamma)\cap\C[x_2,\ldots,x_n]$; here $I(\gamma)$ is the ideal of $\gamma$. 

It holds that $V(I_1(\gamma))$ has dimension 1 and degree at most $b$. We give a proof, adding references for the necessary algebraic geometric background. 

We consider the lexicographical order $\prec$ on the set of monomials in $n$ variables, such that $x_1\succ x_2 \succ \ldots \succ x_n$. With this order, if $G$ is a Gr\"obner basis of $I(\gamma)$, then $G':=G \cap \C[x_2,\ldots,x_n]$ is a Gr\"obner basis of $I_1(\gamma)$ (see \cite{CLO_91}). Note that a Gr\"obner basis $\mathcal{G}$ of an ideal $I$ generates the ideal, while the $\prec$-leading terms of the elements of $\mathcal{G}$ generate the initial ideal $in_{\prec}(I)$ of $I$ (the ideal generated by the $\prec$-leading terms of the elements in $I$).

Now, we see that $\{$monomials in $in_{\prec}(I(\gamma))\cap \C[x_2,\ldots,x_n]\}\subset \{$monomials in $in_{\prec}(I_1(\gamma))\}$. Indeed, let $p$ be a monomial in $in_{\prec}(I(\gamma))$ $\cap$ $\C[x_2,$ $\ldots,x_n]\}$. Since $in_{\prec}(I(\gamma))$ is a monomial ideal, $p$ is divisible by the $\prec$-leading term $in(g)$ of a $g \in G$; so, $in(g)$ is not a multiple of $x_1$, and thus $g \in G\cap \C[x_2,\ldots,x_n]=G'$. Hence, $p$ is divisible by the $\prec$-leading term of an element of $G'$, therefore $p\in in_{\prec}(I_1(\gamma))$.

The dimension of any variety $V$ in $\C^n$ is equal to the cardinality of a maximal subset $S$ of $\{x_1,\ldots,x_n\}$, such that no monomial in the variables in $S$ belongs to the initial (w.r.t. $\prec$) ideal of the ideal of $V$ (see \cite{Stu_05}). Since $\gamma$ is 1-dim, for any $i,j \in \{1,\ldots,n\}$ with $i\neq j$ there exists a monomial $p_{ij}$ in $in_{\prec}(I(\gamma))$ in the variables $x_i$, $x_j$. Now, fix any $i,j \in \{2,\ldots,n\}$ with $i\neq j$. Since $p_{ij}\in in_{\prec}(I(\gamma))\cap \C[x_2,\ldots,x_n]$, it follows by the above discussion that $p_{ij}\in in_{\prec}(I_1(\gamma))$. So, for any $i,j \in \{2,\ldots, n\}$ with $i\neq j$, there exists a monomial in $in_{\prec}(I_1(\gamma))$ in the variables $x_i$, $x_j$. Therefore, $V(I_1(\gamma))$ is at most 1-dim, and thus a curve, due to the genericity of $\Pi$.

It is also known that the degree of a complex algebraic curve $\tilde{\gamma}$ in $\C^n$ is equal to $P(1)$, where $P$ is the univariate polynomial such that the Hilbert series $H_{\tilde{\gamma},\prec}(t)=\sum_{i=1}^{+\infty}d_it^i$ of the curve is equal to $\frac{P(t)}{1-t}$ (where, for all $i$, $d_i$ is the number of monomials in $n$ variables of degree $i$ that are not in the initial (w.r.t. $\prec$) ideal of the ideal of $\tilde{\gamma}$). By the discussion above, for each $i$, the monomials in variables $x_2,\ldots, x_n$ of degree $i$ that are not in $in_{\prec}(I_1(\gamma))$ are also not in $in_{\prec}(I(\gamma))$. Therefore, each coefficient in the Hilbert series of $V(I_1(\gamma))$ is smaller than or equal to the corresponding coefficient in the Hilbert series of $\gamma$; so, $V(I_1(\gamma))$ has degree at most $b$.

\end{proof}

Moreover, we will use the following generalisation of Lemma \ref{firstpolyd} (for $\mathbb{F}=\R$):

\begin{lemma} \label{firstpolyd'} Let $n \geq 2$, $b>0$. Let $J$ be a set of joints formed by a collection $\Gamma$ of irreducible real algebraic curves in $\R^n$, of degree at most $b$. If $J$ lies in the zero set of some non-zero polynomial in $\R[x_1,\ldots,x_n]$, of degree at most $d$, then there exists a curve in $\Gamma$ containing $\lesssim_b d$ points of $J$.
\end{lemma}

The proof of Lemma \ref{firstpolyd'} is completely analogous to the proof of Lemma \ref{firstpolyd}; the only extra element is the use of (the following corollary of) B\'ezout's theorem (e.g., see \cite{MR732620}):

\begin{theorem} \emph{\textbf{(B\'ezout)}} \label{bezout} Let $\gamma$ be an irreducible real algebraic curve in $\R^n$, of degree at most $b$. Let $p\in\R[x_1,\ldots,x_n]$ be a non-zero polynomial such that $\gamma$ does not lie in the zero set of $p$ in $\R^n$. Then, $\gamma$ intersects the zero set of $p$ $\lesssim_{n,b} \deg p$ times.

\end{theorem}

Now, similarly to the situation of joints formed by lines, Lemma \ref{firstpolyd'} implies Lemmas \ref{quilodransimple'} and \ref{quilodranmult'} that follow; Lemma \ref{quilodransimple'} will be used to show our multijoints estimate, while Lemma \ref{quilodranmult'} will imply our generic joints estimate.

\begin{lemma} \label{quilodransimple'} Let $n \geq 2$, $b>0$. Let $\Gamma_1,\ldots,\Gamma_n$ be finite collections of irreducible real algebraic curves in $\R^n$, of degree at most $b$. If $J$ is a subset of $J(\Gamma_1,\ldots,\Gamma_n)$, such that there exists a non-zero polynomial in $\R[x_1,\dots, x_n]$, of degree at most $d$, vanishing on $J$, then $$|J| \lesssim_{n,b} |\Gamma_1\cup\ldots\cup\Gamma_n| \cdot d. $$

\end{lemma}

\begin{lemma} \label{quilodranmult'} Let $n \geq 2$, $b>0$. Let $\Gamma$ be a finite collection of irreducible real algebraic curves in $\R^n$, of degree at most $b$. If, for $k \geq n$, $J$ is a subset of $J^k(\Gamma)$, such that there exists a non-zero polynomial in $\R[x_1,\ldots,x_n]$, of degree at most $d$, vanishing on $J$, then $$|J| k \lesssim_{n,b}|\Gamma| d.$$
\end{lemma}

More precisely, the proof of Theorem \ref{theoremmult3} will be completely analogous to the proof of Theorem \ref{theoremmult2}, which will be adjusted to suit considerations regarding multijoints formed by real algebraic curves instead of lines with the use of  Lemma \ref{projection}, B\'ezout's theorem (Theorem \ref{bezout}), Lemma \ref{quilodransimple'}, as well as the following, which can be found, for example, in \cite{MR2248869}:

\begin{lemma}\label{pathconnected} For every $n \geq 2$ and $b>0$, every irreducible real algebraic curve in $\R^n$, of degree at most $b$, has $\lesssim_{n,b} 1$ path-connected components.
\end{lemma}

Lemma \ref{pathconnected} is a combination of Harnack's theorem for planar curves (see \cite{Harnack}) and the fact that, if $\gamma$ is a real algebraic curve in $\R^n$, of degree at most $b$, then its projection on a generic 2-dim linear subspace of $\R^n$ is contained in a real planar algebraic curve, of degree at most $b$.

Let us now see a sketch of the proof of Theorem \ref{theoremmult3}.\\

\textit{Proof of Theorem \ref{theoremmult3}.} For $n \geq 3$, similarly to the case of multijoints formed by lines, we will show that, if Theorem \ref{theoremmult3} holds in $\R^{n-1}$, then there exists a polynomial in $\R[x_1,\ldots,x_n]$, of low degree, vanishing on a large proportion of our set of multijoints. Then, Lemma \ref{quilodransimple'} will imply that Theorem \ref{theoremmult3} holds in $\R^n$. Since Theorem \ref{theoremmult3} obviously holds in $\R^2$, its proof will be complete.

In particular, if Theorem \ref{theoremmult3} holds in $\R^{n-1}$, for $n\geq 3$, then, for any finite collections $\Gamma_1, \ldots, \Gamma_n$ of irreducible real algebraic curves in $\R^n$, of degree at most $b$, such that $|\Gamma_1|\leq \ldots\leq |\Gamma_n|$, there exists a subset $J$ of $J(\Gamma_1,\ldots, \Gamma_n)$, with
$$|J| \gtrsim |J(\Gamma_1,\ldots,\Gamma_n),
$$
and a non-zero polynomial $p \in \R[x_1,\ldots,x_n]$, with
$$\deg p\lesssim_{n,b} \frac{|\Gamma_1|\cdots|\Gamma_{n-1}|}{|J(\Gamma_1,\ldots,\Gamma_n)|^{n-2}},
$$
vanishing on $J$. To prove this, we apply Guth--Katz partitioning (Theorem \ref{2.1.3}) to the finite set $J(\Gamma_1,\ldots,\Gamma_n)$, with a polynomial $p \in \R[x_1,\ldots,x_n]$ whose degree is at most
$$d:= A_{n,b}\frac{|\Gamma_1|\cdots|\Gamma_{n-1}|}{|J(\Gamma_1,\ldots,\Gamma_n)|^{n-2}},
$$
for an appropriately large constant $A_{n,b}$, depending only on $n$ and $b$. Note that $d>1$, due to our assumption that Theorem \ref{theoremmult3} holds in $\R^{n-1}$. More precisely, the projection of our set of multijoints on a generic hyperplane is a subset of the set of multijoints formed by the collections $P(\Gamma_1),\ldots,P(\Gamma_{n-1})$, where, for each $i=1,\ldots,n-1$, $P(\Gamma_i)$ is the collection of the smallest real algebraic curves, on the hyperplane, containing the projections of the curves in $\Gamma_i$ on the hyperlane (these curves exist and each has degree at most $b$, by to Lemma \ref{projection}. Therefore, applying Theorem \ref{theoremmult3} in $\R^{n-1}$ for the set of multijoints formed by $P(\Gamma_1),\ldots,P(\Gamma_{n-1})$, we see that $d>1$ for $A_{n,b}$ sufficiently large.

Now, $\R^n$ is partitioned in $\sim_n d^n$ cells, each containing $\lesssim_n |J(\Gamma_1,\ldots,\Gamma_n)|/d^n$ points of $J(\Gamma_1,\ldots,\Gamma_n)$. We complete the proof of Theorem \ref{theoremmult3} by showing that $\gtrsim |J(\Gamma_1,\ldots,\Gamma_n)|$ points of $J(\Gamma_1,\ldots,\Gamma_n)$ lie on the zero set $Z$ of $p$. 

Indeed, let us assume for contradiction that $\gtrsim |J(\Gamma_1,\ldots,\Gamma_n)|$ points of $J(\Gamma_1,\ldots,\Gamma_n)$ lie in the union of the cells. Then, there exist $\gtrsim_n d^n$ cells, each containing $\gtrsim_n |J(\Gamma_1,\ldots,\Gamma_n)|/d^n$ points of $J(\Gamma_1,\ldots,\Gamma_n)$ (``full" cells). As in the case of multijoints formed by lines, there exists a full cell $C$ that is intersected by $\lesssim_{n,b} |\Gamma_i|/d^{n-1}$ curves in $\Gamma_i$, for all $i=1,\ldots,n$ (we explain this in the next paragraph); let $\Gamma_{i,C}$ be the set of these curves for each $i=1,\ldots,n$. Then, similarly to the case of multijoints formed by lines, we project the set of multijoints contained in $C$ on a generic hyperplane. The projected points are multijoints formed by $P(\Gamma_{1,C}),\ldots,P(\Gamma_{n-1,C})$, where, for each $i=1,\ldots,n-1$, $P(\Gamma_{i,C})$ is the set of smallest real algebraic curves on the hyperplane that contain the projections of the curves of $\Gamma_{i,C}$ on the hyperplane. By Lemma \ref{projection}, for each $i=1,\ldots,n$, $|P(\Gamma_{i,C})|\lesssim_{n,b}|\Gamma_{i,C}|\lesssim_{n,b}\frac{|\Gamma_i|}{d^{n-1}}$ , and each curve in $P(\Gamma_{i,C})$ has degree at most $b$. Thus, applying Theorem \ref{theoremmult3} in $\R^{n-1}$ for the set of multijoints formed by $P(\Gamma_{1,C}),\ldots,P(\Gamma_{n-1,C})$, we get a contradiction for $A_{n,b}$ large (similarly to the case of multijoints formed by lines). This completes the proof of Theorem \ref{theoremmult3}.

We now conclude by explaining why there exists a full cell that is intersected by $\lesssim_{n,b} |\Gamma_i|/d^{n-1}$ curves in $\Gamma_i$, for all $i=1,\ldots,n$. The reason is that, for each $i=1,\ldots,n$, there exist $\gtrsim_n d^n$ full cells, each of which is intersected by $\lesssim_{n,b} |\Gamma_i|/d^{n-1}$ curves in $\Gamma_i$ (making the implicit constants sufficiently large completes the proof). Indeed, fix $i \in \{1,\ldots,n\}$. Suppose for contradiction that there exist $\lesssim_n d^n$ full cells, each intersected by $\lesssim_{n,b}|\Gamma_i|/d^{n-1}$ curves in $\Gamma_i$. Then, there exist $\gtrsim_n d^n$ full cells, each intersected by $\gtrsim_{n,b}|\Gamma_i|/d^{n-1}$ curves in $\Gamma_i$. Therefore, by Lemma \ref{pathconnected}, there exist $\gtrsim_{n,b} |\Gamma_i|d$ incidences between $Z$ and the curves in $\Gamma_i$ that do not lie in $Z$. However, by B\'ezout's theorem (Theorem \ref{bezout}), there exist $\lesssim_{n,b} |\Gamma_i|d$ such incidences. Being careful with the implicit constants, we have a contradiction.

\qed

The proof of Theorem \ref{genericjoints2} will be analogous to the proof of Theorem \ref{genericjoints}. The new ingredients will be Lemma \ref{projection}, B\'ezout's theorem (Theorem \ref{bezout}), Lemmas \ref{quilodranmult'} and \ref{pathconnected}, as well as ideas from the recent paper \cite{WYZ13} by Wang, Yang and Zhang, including the following Szemer\'edi--Trotter type theorem for real planar algebraic curves:

\begin{theorem}\emph{\textbf{(Wang, Yang, Zhang,\cite{WYZ13})}}\label{wangyangzhang} Let $b$ be a positive integer. Let $\Gamma$ be a finite collection of real algebraic curves in $\R^2$, of degree at most $b$, with no repeated components. Let $\mathcal{P}$ be a finite set of points in $\R^2$. Then,
$$I(\mathcal{P},\Gamma)\lesssim_b |\mathcal{P}|^{\frac{A}{2A-1}}|\Gamma|^{\frac{2A-2}{2A-1}}+|\mathcal{P}|+|\Gamma|,
$$
where $A=\binom{b+2}{2}-1$.
\end{theorem}

\begin{corollary} \label{wangyangzhangcor} Let $b$ be a positive integer, $n \geq 2$. Let $\Gamma$ be a finite collection of real algebraic curves in $\R^2$, of degree at most $b$. Then, for all $k\geq n$,
$$|J^k(\Gamma)|\lesssim_b \frac{|\Gamma|^2}{k^{2+\frac{1}{A-1}}}+\frac{|\Gamma|}{k},
$$
where $A=\binom{b+2}{2}-1$.
\end{corollary}

\begin{proof} Any real algebraic curve in $\R^2$, of degree at most $b$, contains $\lesssim_b 1$ irreducible components; we may therefore assume that the curves in $\Gamma$ are all irreducible. 

We denote by $I^*(J^k(\Gamma),\Gamma)$ the set of incidences between $J^k(\Gamma)$ and $\Gamma$ including multiplicities (i.e., counting each $\gamma\in \Gamma$ through a point $x \in J^k(\Gamma)$ as many times as the number of times $\gamma$ crosses itself at $x$). By the definition of $J^k(\Gamma)$,
\begin{equation}\label{eq:lower} I^*(J^k(\Gamma),\Gamma)\geq |J^k(\Gamma)|\;k.
\end{equation}
On the other hand, each $\gamma \in \Gamma$ crosses itself $\lesssim_b1$ times, therefore
$$I^*(J^k(\Gamma),\Gamma)\lesssim_b I(J^k(\Gamma),\Gamma) + |\Gamma|
$$
\begin{equation} \label{eq:upper}\lesssim_b|J^k(\Gamma)|^{\frac{A}{2A-1}}|\Gamma|^{\frac{2A-2}{2A-1}}+|J^k(\Gamma)|+|\Gamma|;
\end{equation}
note that the last inequality is due to Theorem \ref{wangyangzhang}. Now, combining \eqref{eq:lower} and \eqref{eq:upper} completes the proof.

\end{proof}

We will also use Lemma \ref{dimA} that follows, which is the main idea in \cite{WYZ13}:

\begin{lemma}\emph{\textbf{(Wang, Yang, Zhang, \cite{WYZ13})}}\label{dimA} Let $\gamma$ be a real algebraic curve in $\R^2$, of degree $b$. Let $A:=\binom{b+2}{2}-1$. Then, for every set $\mathcal{P}$ of $b^2+1$ points of $\gamma$, there exists a subset $\mathcal{P}'$ of $\mathcal{P}$, consisting of $A$ points, such that, if $\gamma'$ is a real algebraic curve in $\R^2$, of degree at most $b$, passing through each point of $\mathcal{P}'$, then $\gamma'$ and $\gamma$ have a common component.

\end{lemma}

More particularly, it is known by B\'ezout's theorem that, if $\Gamma$ is a set of real planar algebraic curves of degree at most $b$, with no common components, then every set of $b^2+1$ points of a curve in $\Gamma$ fully determines that curve in $\Gamma$. However, it is shown in \cite{WYZ13} that each set of $b^2+1$ points of a curve in $\Gamma$ has some subset of $\binom{b+2}{2}-1$ points (a strict subset for $b \geq 3$) that also fully determines the curve in $\Gamma$. This is proved in \cite{WYZ13} using the fact that the vector space $\R_{b}[x,y]$ of polynomials in $\R[x,y]$ of degree $\leq b$ has dimension $\binom{b+2}{2}$, and thus the number of linearly independent conditions that $\mathcal{P}$ imposes on $\R_b[x,y]$, is $\leq \binom{b+2}{2}-1$. \\

\textit{Proof of Theorem \ref{genericjoints2}.} For $n \geq 3$, similarly to the case of generic joints formed by lines, we will show that, if Theorem \ref{genericjoints2} holds in $\R^{n-1}$, then, under the assumption that our set of generic joints is large, there exists a polynomial in $\R[x_1,\ldots,x_n]$, of low degree, vanishing on a large proportion of our set of joints. Then, Lemma \ref{quilodranmult'} will imply that Theorem \ref{genericjoints2} holds in $\R^n$. Since Theorem \ref{genericjoints2} holds in $\R^2$ (by Corollary \ref{wangyangzhangcor}), its proof will be complete.

In particular, if Theorem \ref{genericjoints2} holds in $\R^{n-1}$, for $n \geq 3$, then, for any $k \geq n$ and any finite collection $\Gamma$ of real algebraic curves in $\R^n$ of degree at most $b$, such that $$|J^k(\Gamma)|\geq C_{n,b} \left(\frac{|\Gamma|^{\frac{n}{n-1}}}{k^{\frac{n}{n-1}+\frac{1}{n-1}\cdot\frac{1}{B-1}}}+\frac{|\Gamma|}{k}\right)$$ for some sufficiently large constant $C_{n,b}$, depending only on $n$ and $b$, there exists a subset $J^k$ of $J^k(\Gamma)$, with 
$$|J^k| \gtrsim |J^k(\Gamma)|,$$
and a non-zero polynomial $p \in \R[x_1,\ldots,x_n]$, with
$$\deg p \leq A_{n,b} \frac{|\Gamma|^{n-1}}{|J^k(\Gamma)|^{n-2}k^{n-1+\frac{1}{B-1}}},$$
vanishing on $J^k$, where $A_{n,b}$ is a constant depending only on $n$ and $b$. We will prove this by applying Guth--Katz partitioning (Theorem \ref{2.1.3}) to the finite set $J^k(\Gamma)$, with a polynomial $p \in \R[x_1,\ldots,x_n]$ whose degree is at most
$$d:= A_{n,b}\frac{|\Gamma|^{n-1}}{|J^k(\Gamma)|^{n-2}k^{n-1+\frac{1}{B-1}}},
$$
for an appropriately large constant $A_{n,b}$, depending only on $n$ and $b$. 

Note that $d>1$ for $A_{n,b}$ sufficiently large, due to our assumption that Theorem \ref{genericjoints2} holds in $\R^{n-1}$. Let us explain this in detail. For each $\gamma \in \Gamma$, we define $P^{n-1}(\gamma)$ to be the smallest real algebraic curve, on a generic hyperplane of $\R^n$, containing the projection of $\gamma$ on the hyperplane; such a curve exists and has degree at most $b$ (by Lemma \ref{projection}). Let $P^{n-1}(\Gamma):=\{P^{n-1}(\gamma):\gamma\in \Gamma\}$, and let $P^{n-1}(J^k(\Gamma))$ be the projection of $J^k(\Gamma)$ on the generic hyperplane. It is clear that $P^{n-1}(J^k(\Gamma))\subset \cup_{\lambda \geq k}\;J^{\lambda}(P^{n-1}(\Gamma))$, while $|P^{n-1}(\Gamma)|\leq|\Gamma|$. Thus, applying Theorem \ref{genericjoints2} in $\R^{n-1}$ for $\cup_{\lambda \geq k}\;J^{\lambda}(P^{n-1}(\Gamma))$ $(= J^{k}(P^{n-1}(\Gamma))\cup J^{2k}(P^{n-1}(\Gamma))\cup J^{2^2k}(P^{n-1}(\Gamma))\cup \ldots)$ we obtain
$$|J^k(\Gamma)|\lesssim_{n,b}\sum _{\mu=0}^{+\infty}|J^{2^\mu k}(P^{n-1}(\Gamma))|$$
$$\lesssim_{n,b}\sum_{\mu=0}^{+\infty}\Bigg(\frac{|P^{n-1}(\Gamma)|^{\frac{n-1}{n-2}}}{(2^{\mu}k)^{\frac{n-1}{n-2}+\frac{1}{n-2}\cdot\frac{1}{B-1}}}+\frac{|P^{n-1}(\Gamma)|}{2^{\mu}k}\Bigg)
$$
$$\lesssim_{n,b}\frac{|\Gamma|^{\frac{n-1}{n-2}}}{k^{\frac{n-1}{n-2}+\frac{1}{n-2}\cdot\frac{1}{B-1}}}+\frac{|\Gamma|}{k}.
$$
Now, given our assumption that
$$|J^k(\Gamma)|\geq C_{n,b} \frac{|\Gamma|}{k},
$$
we can assume that $C_{n,b}$ is large enough to have
$$\frac{|\Gamma|}{k} \lesssim_{n,b} \frac{|\Gamma|^{\frac{n-1}{n-2}}}{k^{\frac{n-1}{n-2}+\frac{1}{n-2}\cdot\frac{1}{B-1}}},
$$
and therefore that
$$|J^k(\Gamma)| \lesssim_{n,b}\frac{|\Gamma|^{\frac{n-1}{n-2}}}{k^{\frac{n-1}{n-2}+\frac{1}{n-2}\cdot\frac{1}{B-1}}},
$$
which indeed means that $d>1$ for $A_{n,b}$ sufficiently large.

In this way, $\R^n$ is partitioned in $\sim_n d^n$ cells, each containing $\lesssim_n |J^k(\Gamma)|/d^n$ points of $J^k(\Gamma)$. We complete the proof of Theorem \ref{genericjoints2} by showing that $\gtrsim |J^k(\Gamma)|$ points of $|J^k(\Gamma)|$ lie on the zero set $Z$ of $p$. 

Indeed, let us assume that $\gtrsim |J^k(\Gamma)|$ points of $J^k(\Gamma)$ lie in the union of the cells. Then, there exist $\gtrsim_n d^n$ cells, each containing $\gtrsim_n |J^k(\Gamma)|/d^n$ points of $J^k(\Gamma)$ (``full" cells). It follows (similarly to the case of multijoints formed by real algebraic curves) that there exists a full cell $C$ that is intersected by $\lesssim_{n,b} |\Gamma|/d^{n-1}$ elements of $\Gamma$ (note that here the curves in $\Gamma$ are not necessarily real algebraic curves, but they are contained in real algebraic curves). 

In analogy to the situation of generic joints formed by lines, we split in two cases: the case where $C$ contains $\lesssim_{n,b} k^{\frac{1}{B-1}}$ points of $J^k(\Gamma)$, and the case where it contains $\gtrsim_{n,b} k^{\frac{1}{B-1}}$ points of $J^k(\Gamma)$.

If $C$ contains $\leq \lambda_{n,b}\; k^{\frac{1}{B-1}}$ points of $J^k(\Gamma)$, for some constant $\lambda_{n,b}$ which will be specified later and depends only on $n$ and $b$ , then
$$\frac{|J^k(\Gamma)|}{d^n} \lesssim_{n,b} k^{\frac{1}{B-1}},
$$
and thus
$$|J^k(\Gamma)| \leq C'_{n,b}\frac{|\Gamma|^{\frac{n}{n-1}}}{k^{\frac{n}{n-1}+\frac{1}{n-1}\cdot\frac{1}{B-1}}},
$$
for some constant $C'_{n,b}$ independent of $C_{n,b}$; by fixing $C_{n,b}\gneq C'_{n,b}$, we have a contradiction. 

It therefore follows that $C$ contains $\geq \lambda_{n,b}\;k^{\frac{1}{B-1}}$ points of $J^k(\Gamma)$. Let $\Gamma_{C}$ be the subcollection of $\Gamma$ that consists of the curves in $\Gamma$ that intersect $C$. Similarly to the case of generic joints formed by lines, we project $J^k(\Gamma)\cap C$ on a generic hyperplane; let $P^{n-1}(J^k(\Gamma)\cap C)$ be the set of these projected points. Moreover, let $P^{n-1}(\Gamma_{C}):=\{P^{n-1}(\gamma):\gamma\in\Gamma_C\}$. By Lemma \ref{projection}, each curve in $P^{n-1}(\Gamma_C)$ has degree at most $b$. It is clear that $P^{n-1}(J^k(\Gamma)\cap C)\subset \cup_{\lambda \geq k}\;J^{\lambda}(P^{n-1}(\Gamma_C))$ and $|P^{n-1}(\Gamma_C)|\leq |\Gamma_C|$. Thus, applying Theorem \ref{genericjoints2} in $\R^{n-1}$ for $\cup_{\lambda \geq k}\;J^{\lambda}(P^{n-1}(\Gamma_C))$ $(=J^{k}(P^{n-1}(\Gamma_C))\cup J^{2k}(P^{n-1}(\Gamma_C))\cup J^{2^2k}(P^{n-1}(\Gamma_C))\cup \ldots)$ we obtain
$$|J^k(\Gamma)\cap C| \leq \sum_{\mu=0}^{+\infty}|J^{2^{\mu}k}(P^{n-1}(\Gamma_C))| $$
$$\lesssim_{n,b} \sum_{\mu=0}^{+\infty}\Bigg(\frac{|P^{n-1}(\Gamma_{C})|^{\frac{n-1}{n-2}}}{(2^{\mu}k)^{\frac{n-1}{n-2}+\frac{1}{n-2}\cdot\frac{1}{B-1}}}+\frac{|P^{n-1}(\Gamma_{C})|}{2^{\mu}k}\Bigg)
$$
\begin{equation}\label{eq:conclude}\lesssim_{n,b}\frac{|\Gamma_{C}|^{\frac{n-1}{n-2}}}{k^{\frac{n-1}{n-2}+\frac{1}{n-2}\cdot\frac{1}{B-1}}}+\frac{|\Gamma_{C}|}{k}.
\end{equation}
Note that $|J^k(\Gamma)\cap C|\gtrsim_{n,b} \frac{|J^k(\Gamma)|}{d^n}$ and $|\Gamma_{C}|\lesssim_{n,b}\frac{|\Gamma|}{d^{n-1}}$; therefore, if we show that 
\begin{equation} \label{eq:end}\frac{|\Gamma_{C}|}{k}\lesssim_{n,b}\frac{|\Gamma_{C}|^{\frac{n-1}{n-2}}}{k^{\frac{n-1}{n-2}+\frac{1}{n-2}\cdot\frac{1}{B-1}}},
\end{equation}
then \eqref{eq:conclude} will imply that $d \lesssim_{n,b}\frac{|\Gamma|^{n-1}}{|J^k(\Gamma)|^{n-2}k^{n-1+\frac{1}{B-1}}}$. By fixing $A_{n,b}$ to be sufficiently large we will have a contradiction, and the proof of Theorem \ref{genericjoints2} will be complete.

We conclude by showing \eqref{eq:end}, in other words that $|\Gamma_{C}| \gtrsim_{n,b} k^{\frac{B}{B-1}}$. Note that here we will use an approach different to the one applied in the case of generic joints formed by lines; this approach is essentially the proof of the `trivial bound' in \cite{WYZ13}. Let $J_{C}$ be a subset of $J^k(\Gamma)\cap C$, such that
\begin{equation}\label{eq:i}|J_{C}|=\lambda_{n,b}\; k^{\frac{1}{B-1}}.
\end{equation}
For each $\gamma\in \Gamma_{C}$, we define $P^2(\gamma)$ to be the smallest real planar algebraic curve containing the projection of $\gamma$ on a generic 2-dim linear subspace of $\R^n$; moreover, for each $x \in J_C$, let $P^2(x)$ be the projection of $x$ on the subspace. Let $P^2(\Gamma_{C}):=\{$irreducible components of $P^2(\gamma): \gamma\in \Gamma_C\}$, and $P^2(J_C):=\{P^2(x): x \in J_C\}$. Since each point $x \in J_{C}$ belongs to $J^k(\Gamma_{C})$, there should exist at least $k$ curves in $P^2(\Gamma_C)$ through $P^2(x)$ counted with multiplicity (i.e., counting each curve as many times as the number of times it crosses itself at $P^2(x)$). Therefore, since any curve in $P^2(\Gamma_C)$ has degree at most $b$ (by Lemma \ref{projection}) and thus crosses itself $\lesssim_{b,n}1$ times, we have that
$$|J_C|\cdot k \lesssim_{n,b} I(P^2(J_C),P^2(\Gamma_C))+|P^2(\Gamma_C)|
$$ 
\begin{equation}\label{eq:iii}\lesssim_{n,b} I(P^2(J_C),P^2(\Gamma_C))+|\Gamma_C|;
\end{equation}
note that the last inequality is due to the fact that $|P^2(\Gamma_C)|\lesssim_{n,b} |\Gamma_C|$.
Now, for each $\gamma \in P^2(\Gamma_C)$, we denote by $N_{\gamma}$ the set of points of $P^2(J_C)$ on $\gamma$. Since the curves in $P^2(\Gamma_{C})$ containing $\leq b^2$ points of $P^2(J_{C})$ contribute at most $b^2\cdot |P^2(\Gamma_{C})|$ incidences with $P^2(J_{C})$, it holds that 
$$ I(P^2(J_C),P^2(\Gamma_C)) \leq b^2\cdot |P^2(\Gamma_{C})|+ \sum_{\left\{\gamma \in P^2(\Gamma_{C}):|N_{\gamma}|\geq b^2+1\right\}}|N_{\gamma}|
$$
\begin{equation}\label{eq:iv}\lesssim_{n,b}|\Gamma_{C}|+ \sum_{\left\{\gamma \in P^2(\Gamma_{C}):|N_{\gamma}|\geq b^2+1\right\}}|N_{\gamma}|.
\end{equation}
Let $\gamma \in P^2(\Gamma_C)$ with $|N_{\gamma}|\geq b^2+1$. Each element of $P^2(\Gamma_{C})$ is an irreducible real planar algebraic curve of degree at most $b$, thus, by Lemma \ref{dimA}, each set of $b^2+1$ points in $N_{\gamma}$ has a subset of cardinality $B$ that fully determines $\gamma$ in $P^2(\Gamma)$. Therefore, since each such $B$-tuple that fully determines $\gamma$ can be the same for at most $\binom{|N_{\gamma}|-B}{b^2+1-B}$ $(b^2+1)$-tuples of points in $N_{\gamma}$, there exist at least $\binom{|N_{\gamma}|}{b^2+1}/\binom{|N_{\gamma}|-B}{b^2+1-B}\gtrsim_{n,b}|N_{\gamma}|^{B}\gtrsim_{n,b}|N_{\gamma}|$ $B$-tuples of points in $N_{\gamma}$ that fully determine $\gamma$. On the other hand, all the $B$-tuples of points in $P^2(J_C)$ are $\lesssim_{n,b}|P^2(J_C)|^{B}\sim_{n,b} |J_C|^{B}$ in total, so
\begin{equation}\label{eq:v}\sum_{\left\{\gamma \in P^2(\Gamma_{C}):|N_{\gamma}|\geq b^2+1\right\}}|N_{\gamma}|\lesssim_{n,b}|J_C|^{B}.
\end{equation}
It follows by \eqref{eq:iii}, \eqref{eq:iv} and \eqref{eq:v} that
$$|J_C|\cdot k \leq c'_{n,b} \; (|\Gamma_{C}|+|J_C|^{B}),
$$
for some constant $c'_{n,b}$, depending only on $n$ and $b$. Combining this with \eqref{eq:i} we obtain
$$\lambda_{n,b}\;k^{\frac{B}{B-1}}\leq c'_{n,b} \;|\Gamma_C|+c'_{n,b}\;\lambda_{n,b}^{B}\;k^{\frac{B}{B-1}}.
$$
Thus, by fixing $\lambda_{n,b}$ to be appropriately small, we get
$$|\Gamma_{C}| \gtrsim_{n,b} k^{\frac{B}{B-1}},$$
which completes the proof.

\qed



\end{document}